\documentclass[a4paper]{article}
\usepackage[english,french]{babel}
\usepackage{latexsym}
\usepackage{amsmath}
\usepackage{amssymb}
\usepackage{graphics}
\usepackage{graphicx}
\usepackage{fancyhdr}
\usepackage{color}
\usepackage{latexsym}
\usepackage{epstopdf}
\usepackage{amsfonts}
\usepackage[utf8]{inputenc}
\usepackage[numbers]{natbib}				% bibliographie
\usepackage{fullpage}
\renewcommand{\cite}[1]{\citeauthor*{#1} [\citeyear{#1}]}

\usepackage{amsthm,url,xspace}
\definecolor{couleurCitations}{rgb}{0,0.65,0}
\definecolor{couleurRef}{rgb}{0.75,0,0}
\usepackage[pdftex,colorlinks=true,
linkcolor=couleurRef,citecolor=blue,a4paper=true
,bookmarks=true,plainpages=true,
 urlcolor= blue]{hyperref}%backref=page,pagebackref=true,

\title{\textbf{Adaptive Noisy Clustering}}
\author{Micha\"el CHICHIGNOUD\footnote{ Seminar for Statistics, ETH Z\"urich, R\"amistrasse 101, CH-8092 Z\"urich, SWITZERLAND} \, and S\'ebastien LOUSTAU\footnote{ LAREMA, Universit\'e d'Angers, 2 Bvd Lavoisier 49045 Angers Cedex, FRANCE}}

\date\today
\makeatletter

\@addtoreset{equation}{section}
\makeatother

\newcommand{\epr}{\hfill\hbox{\hskip 4pt
                \vrule width 5pt height 6pt depth 1.5pt}\vspace{0.5cm}\par}

\newcommand{\cB}{{\cal B}}
\newcommand{\cC}{{\cal C}}

\newcommand{\cF}{{\cal F}}
\newcommand{\cG}{{\cal G}}

\newcommand{\cK}{{\cal K}}

\newcommand{\cM}{{\cal M}}

\newcommand{\cR}{{\cal R}}

\newcommand{\cX}{{\cal X}}

\newcommand{\cZ}{{\cal Z}}

\newcommand{\bE}{\mathbb E}

\newcommand{\bN}{{\mathbb N}}

\newcommand{\bP}{{\mathbb P}}

\newcommand{\bR}{{\mathbb R}}

\newcommand{\bZ}{{\mathbb Z}}
\def\sino{\frac{1}{n}\sum_{i=1}^n}

\def \beqn {\begin{eqnarray}}
\def \beqnn {\begin{eqnarray*}}
\def \eeqn {\end{eqnarray}}
\def \eeqnn {\end{eqnarray*}}

\def \no {\Arrowvert}
\def \ind {\hbox{ 1\hskip -3pt I}}
\def \R {\mathbb{R}}

\def \P  {\mathbb{P}} 
\def \R {\mathbb{R}}
\def \E {\mathbb{E}}
\def \N {\mathbb{N}}

\newcommand{\op}[1]{\operatorname{#1}}			% Lettre droite!!!!

\newcommand{\noyaudeconvolution}[2]{\cK_{#1}\left(#2\right)}

\newcommand{\constseuil}{A}

\newcommand{\ordreterme}[3]{r_{#1}^*\left(#2,#3\right)}

\newcommand{\constbias}{\zeta_1}

\newcommand{\constsupboundprocess}{\zeta_2}

\newcommand{\constloustau}{\zeta_3}

\newcommand{\constminimax}{C_1}

\newcommand{\constadaptative}{C_2}

\newcommand{\threshold}[1]{\delta_{#1}}

\newcommand{\estimator}[1]{\hat{\bc}_{#1}}

\newcommand{\bc}{\mathbf{c}}

\newcommand{\constantAtwo}{T(d)}

\newcommand{\constantsupconvolution}{\eta_{\infty}}

\newcommand{\constantFourierPolynomial}{\rho}

\newcommand{\constantmargin}{\kappa}

\theoremstyle{plain}
\newtheorem{theorem}{Theorem}%[section]
\newtheorem{proposition}{Proposition}%[section]
\newtheorem{lemma}{Lemma}
\newtheorem{definition}{Definition}

\newcommand{\1}{{\rm 1}\mskip -4,5mu{\rm l} }

\newcommand{\argmin}{\mathop{\mathrm{arg\,min}}}

\newcommand{\expectationZ}{\bE}

\newcommand{\probaX}{P_X}

\newcommand{\probaZ}{P_Z}

\newcommand{\empiricalproba}{P_n}

\definecolor{orange}{rgb}{0.8,0.8,0.8}

\newcommand{\nno}{|\hspace{-0.05cm}\no}

\begin{document}
\selectlanguage{english}

\maketitle

\begin{abstract}%   <- trailing '%' for backward compatibility of .sty file
The problem of adaptive noisy clustering is investigated. Given a set of noisy observations $Z_i=X_i+\epsilon_i$, $i=1, \ldots, n$, the goal is to design clusters associated with the law of $X_i$'s, with unknown density $f$ with respect to the Lebesgue measure. Since we observe a corrupted sample, a direct approach as the popular {\it $k$-means} is not suitable in this case. In this paper, we propose a noisy $k$-means minimization, which is based on the $k$-means loss function and a deconvolution estimator of the density $f$. In particular, this approach suffers from the dependence on a bandwidth involved in the deconvolution kernel.  Fast rates of convergence for the excess risk are proposed for a particular choice of the bandwidth, which depends on the smoothness of the density $f$. 

Then, we turn out into the main issue of the paper: the data-driven choice of the bandwidth. We state an adaptive upper bound for a new selection rule, called ERC (Empirical Risk Comparison). This selection rule is based on the Lepski's principle, where empirical risks associated with different bandwidths are compared. Finally, we illustrate that this adaptive rule can be used in many statistical problems of $M$-estimation where the empirical risk depends on a nuisance parameter.
\end{abstract}
{\small
\paragraph{keywords:}
{Adaptivity,} {deconvolution}, {errors-in-variables}, {fast rates}, {M-estimation},
{Statistical learning}.
}
%%\input{notation.tex}
%\textred{Remarques Générales: Pour ne pas se perdre, on note $ \bE^{Z},P_n^Z,P^Z  $ et $ \bE^{X},P_n^X,P^X $ pour les espérances, mesures empiriques et mesures de proba respectives des lois de $Z$ et $X$. De plus, on note $\E$ et $\P$ l'espérance et la proba par rapport à la loi de l'échantillon, i.e. $(P^{Z})^{\otimes n}$.} \\
%\\
%
%\textred{
%Pour le prix à payer, il est clair que, dans notre preuve, on paye un prix $\log(n)$. En revanche, il semble qu'on fait de l'estimation globale de la densité des X inconnues, puisque le risque est intégré sous la mesure de proba de X (ce qui est en fait une norme). Donc en fait, il est peut être possible d'éviter de payer $\log(n)$. Pour cela, je pense qu'il ne faut pas sortir le sup sur $ \lambda $ (transformé en sum) comme on fait tout le temps, il faut plutôt contrôler les déviations du processus avec le sup sur $g$ ET sur $ \lambda $. Peut être que le papier de Lepski et Goldenshluger (Uniform bounds for norms of sums of independent random functions) peut nous aidé à faire ça, à compiner avec la méthode du point fixe.}
%\\

\section{Introduction}

\paragraph{Motivation.}
Nonparametric procedures of estimation contain some nuisance parameter(s) whose optimal selection is not obvious. From the minimax point of view, these methods reach optimal rates of convergence, based on a regularity assumption over the unknown function to estimate. As a consequence, optimal parameters depend on some unknown smoothness index $s>0$ of some functional space (e.g. the Hölder smoothness). In density estimation, the most popular technique of kernel estimators (see \cite{Rosenblatt56} or \cite{Parzen62}) suffers from the dependence on a bandwidth parameter $\lambda>0$. In deconvolution as well, kernel deconvolution estimators are rather popular to estimate a density from a sequence of independent and identically distributed (i.i.d.) contaminated observations:
\beqn
\label{noisydata0}
Z_i=X_i+\epsilon_i, \mbox{ for }i=1,\dots,n,
\eeqn
where $f$ denotes the unknown density of the i.i.d. sequence $X_1,X_2,...,X_n$ and $ \eta $ is the known density of the i.i.d. random variables $\epsilon_1,\epsilon_2,...,\epsilon_n$, independent of the $X_i$'s. In this framework, a {\it kernel deconvolution estimator} of $f$ is given by:
\beqn
\label{decest}
\hat{f}_\lambda(x)=\frac{1}{n}\sum_{i=1}^n\mathcal{K}_\lambda\left(Z_i-x\right),
\eeqn
where $\mathcal{K}_\lambda$ is a deconvolution kernel and $\lambda>0$ is a bandwidth. Estimators of the form \eqref{decest} are of first interest in this paper. \cite{fan} has proved \textit{minimax rates of convergence} for \eqref{decest} for an optimal value of the bandwidth. This deterministic choice trades off a bias term and a variance term, and depends on unknown parameters, such as the smoothness index $s>0$ of the density $f$. 

From the adaptive point of view, the aim is the data-driven selection of an estimator from a given family which has some \textit{adaptive optimal properties}: the selected estimator reaches the minimax rate for any function in a vast range of regularities $s\in]0,s^+[,~s^+>0$. In this case, the proposed estimator does not depend on the exact smoothness index $s>0$ of the target function but only on the upper bound $s^+$. It reaches minimax adaptivity with respect to the unknown smoothness. One of the most popular method for choosing the bandwidth is suggested by \cite{Lepski_Mammen_Spokoiny97} in a gaussian white noise model. It is based on the {\it Lepski's principle} (\cite{Lepski90}). The idea is to test several estimators (by comparison) for different values of the bandwidth. This work is at the origin of a quantity of theoretical papers dealing with minimax adaptive bounds in nonparametric estimation (see for instance \cite{
Goldenshluger_Nemirovski97}, \cite{Mathe05}, \cite{Chichignoud12}). From the practical point of view, Lepski's method has also received further development, such as the intersection of confidence intervals (ICI) rule (see \cite{Katkovnik99}). This algorithm reveals computational advantages in comparison to the traditional Lepski's procedure, or even traditional cross-validation techniques since it does not require to compute all the estimators in the family. This algorithm was originally design for a problem of gaussian filtering. It is at the core of many applications in image processing (see \cite{Kervrann_Boulanger06}, \cite{Katkovnik_Foi_Egiazarian_Astola10} and references therein). In the context of deconvolution, \cite{Comte_Lacour13} gets adaptive optimal results (for pointwise and global risks) using an improvement of the standard Lepski's principle (see \cite{Goldenshluger_Lepski11}).\\

\paragraph{Noisy data.} In the present paper, we deal with the problem of clustering with noisy observations \eqref{noisydata0}. Classical results in the presence of noisy observations are given in a deconvolution framework (see the references above), or alternatively in regression with errors-in-variables. \cite{fantruong} gives for the first time the minimax rates in the model of regression with errors-in-variables. \cite{delaiglehallmeister} studies both density deconvolution and regression with noisy data when the noise density $\eta$ is unknown. We mention the monograph of \cite{meister} for a complete survey, including cross-validation techniques to choose the bandwidth in regression with noisy measurements. In statistical learning, \cite{koltgeometry} proposes to study different geometric characteristics of a multivariate distribution, such as the entropy dimension or the number of clusters, thanks to noisy data of the form \eqref{noisydata0}. In a quiet related framework, minimax results in Hausdorff 
distance are stated in \cite{mme} for manifold estimation of the support of a distribution thanks to noisy data. 

More recently, \cite{pinkfloyds} proposes to study a model of classification with noisy data, by giving for the first time minimax rates in binary classification with errors in variables. This paper is at the origin of other works in statistical learning with noisy data (see \cite{Loustau12} for supervised classification and \cite{Loustau12bis} for unsupervised problems). In these problems, the use of a deconvolution kernel estimator \eqref{decest} is necessary to derive \textit{excess risk bounds}. \cite{pinkfloyds} suggests a deterministic bandwidth choice to get {minimax {\it fast} convergence rates} (i.e. faster than $n^{-1/2}$). Unfortunately, as usually, this choice depends on the unknown smoothness of the density $f$. 
\paragraph{Outlines.} 
The aim of this contribution is to get adaptive fast rates for the excess risk via a new Lepski-type procedure. To the best of our knowledge, standard adaptive procedures such as cross-validation, model selection or aggregation cannot be directly applied in our particular context (see Section \ref{sec:adaptive} for details). Moreover, the Lepski's principle, which is usually used by comparing estimators for a given pointwise or global risk, cannot be directly applied to get excess risk bounds. 
%Some works have been done to extend the procedure to local likelihood comparison, using Kullback divergence as risk  (see \cite{Polzehl_Spokoiny06}). 
% More recently, \cite{brunel} extend Lepski's principle to the comparison of set estimators, using the Lebesgue measure of the symmetric difference as criterion. 
In this contribution, we design a new selection rule based on the Lepski's principle with a comparison of {\it empirical risks} with different nuisance parameters. This method, called {\it Empirical Risk Comparison} (ERC), allows us to derive adaptive results in the context of clustering with noisy data. 
%It could be seen as a generalization of \cite{Polzehl_Spokoiny06} to the general setting of $M$-estimation (see Section \ref{sec:gen}).
It could be applied to the general setting of $M$-estimation to derive adaptive results in other statistical problems.

The paper is organized as follows. In Section \ref{sec:NC}, we describe the model and the empirical risk minimization called {\it noisy $k$-means}, which uses a collection of deconvolution estimators \eqref{decest} to deal with noisy data. In Section \ref{sec:nonadaptive}, we  state a non-adaptive risk bound, under a smoothness assumption over the density $f$, and an ill-posedness assumption over the noise distribution. These rates are reached by the noisy $k$-means procedure, where the bandwidth $\lambda>0$ is chosen to trade off a bias-variance decomposition. 
%As in standard nonparametric estimation, this optimal value for $\lambda$ is non-adaptive since it depends on the ad-hoc assumptions. 
In Section \ref{sec:adaptive}, we present the adaptive procedure ERC to choose the bandwidth automatically. The theoretical results guarantee the same rates of convergence, modulo an extra-log term which seems to be optimal.  Finally, we conclude in Section \ref{sec:gen} by a generalization of the ERC selection rule to other $M$-estimation problems, such as binary classification, local M-estimation or quantile estimation. Section \ref{sec:conclusion} concludes the paper whereas Section \ref{sec:proofs}-\ref{sec:appendix} are dedicated to the proofs of the main results.

\section{Noisy Clustering} 
\label{sec:NC}
\subsection{The problem}
Isolate meaningfull groups from the data is an interesting topic in data analysis with applications in many fields, such as biology or social sciences. However, in many real-life situations, direct data are not available and measurement errors occur. Then, we observe a corrupted sample of i.i.d. observations:
\beqn
\label{noisydata}
Z_i=X_i+\epsilon_i,~i=1,\ldots n,
\eeqn 
where $f$ denotes the unknown density of the i.i.d. sequence $X_1,X_2,...,X_n$ and $ \eta $ is the known density of the i.i.d. random variables $\epsilon_1,\epsilon_2,...,\epsilon_n$, independent of the $X_i$'s. In the sequel, the law of the sequence $(X_i)_{i=1,\ldots ,n}$ is denoted as $\probaX$, with density $f$ with respect to the Lebesgue measure on $\bR^d$. We also assume that $X_i$, $i=1,\ldots, n$ are contained in $\mathcal{B}(0,1)$, the unit Euclidean ball of $\bR^d$. This assumption is rather classical in clustering (see \cite{lbclustering}, \cite{levrard}) and known as the {\it peak power constraint} (extension to $\mathcal{B}(0,M)$ with $M>1$ is straightforward). Given some integer $k\geq 1$, 
the problem of noisy clustering is to learn $k$ clusters from $\probaX$ when a contaminated empirical version $Z_1,\ldots ,Z_n$ is observed. This problem is a particular case of inverse statistical learning and is known to be an inverse problem (see \cite{Loustau12}). It has been studied recently in \cite{Loustau12bis}, where non-adaptive  results are proposed.

For this purpose, we introduce a set of codebooks $\bc=(c_1,\dots,c_k)\in\R^{dk}$, and the standard $k$-means loss function $\gamma(\bc,X):=\min_{j=1,\dots,k}\no X-c_j\no^2$, where $ \no\cdot\no $ stands for the Euclidean norm on $ \bR^d $. The corresponding clustering risk of a codebook $\bc$ is given by:
\begin{equation}\label{def risk}
R(\bc):=\E_{\probaX}\gamma(\bc,X)=\int_{\bR^d}\gamma(\bc,x)f(x)dx.
\end{equation}
Given \eqref{def risk}, we measure the performance of the latter codebook $\bc$ in terms of excess risk, defined as:
\begin{equation}\label{def excess risk}
R(\bc,\bc^*):=R(\bc)-R(\bc^*),
\end{equation}
where $ \bc^*\in\arg\min_{\bc\in\cB(0,1)} R(\bc)$ is called an {\it oracle}. The oracle set is denoted as $\mathcal{M}$ and we assume in the rest of the paper that the number $|\mathcal{M}|$ of oracles is finite. This assumption is satisfied in the context of Pollard's regularity assumptions (see \cite{pollard82}), i.e. when $f$ has a continuous density (w.r.t. the Lebesgue measure) such that the Hessian matrix of $\bc\mapsto  R(\bc)$ is positive definite. In the direct case, the problem of minimizing \eqref{def excess risk} has been investigated in a variety of areas. For a given number of clusters $k\geq 1$, the most popular technique is the $k$-means procedure. It consists in partitioning the dataset $X_1,\ldots , X_n$ into $k$ clusters by minimizing the empirical risk:
$$
R_n(\bc)=\frac{1}{n}\sum_{i=1}^n\min_{j=1,\ldots, k}\no X_i-c_j\no^2,
$$
where $\bc=(c_1,\ldots, c_k)\in\R^{dk}$ is a set of centers. A cluster is associated to each observation by giving its nearest center $c_j$, $j=1,\ldots ,k$. The $k$-means clustering minimization has been widely studied in the literature. Since the early work of Pollard (\cite{pollard81},\cite{pollard82}), consistency and rates of convergence have been considered by many authors. \cite{biau} suggests rates of convergence of the form $\mathcal{O}(1/\sqrt{n})$ whereas \cite{lbclustering} proposes a complete minimax study. More recently, \cite{levrard} states fast rates of the form $\mathcal{O}(1/n)$ under Pollard's regularity assumptions. It improves a previous result of \cite{gg}.

However, in this paper, the problem is the knowledge of $X_1,\ldots, X_n$ since we deal with a noisy dataset \eqref{noisydata}. For this reason, we introduce a deconvolution step in the stochastic minimization of the $k$-means procedure.

\subsection{The noisy $k$-means minimization}
Following \cite{Loustau12bis}, the idea is to plug a deconvolution kernel estimator of the form  \eqref{decest} into the true risk \eqref{def risk}. For this purpose, let us introduce the following notations. We denote by $\mathcal{F}[g]$ the Fourier transform of an integrable function $g$, whereas $\mathcal{F}^{-1}$ stands for the inverse Fourier transform. Let $\mathcal{K}$ be a kernel in $L_2(\bR^d)$ such that $\mathcal{F}[\mathcal{K}]$ exists. Then, provided that $\mathcal{F}[\eta]$ exists and is strictly positive, we can introduce a deconvolution kernel $\mathcal{K}_\lambda$ as follows:
\begin{align}\label{def deconvolution kernel}
\mathcal{K}_\lambda\::\:&\bR^d\rightarrow\bR\nonumber\\
& t\mapsto \lambda^{\text-d}\cF^{-1}\left[\frac{\cF[\mathcal{K}](\cdot)}{\cF[\eta](\cdot/\lambda)}\right](t/\lambda),
\end{align}
where $ \lambda>0 $ is called the bandwidth. The kernel $ \mathcal{K} $ in \eqref{def deconvolution kernel} is a kernel with particular properties (see Section \ref{assumptions}). Note that with a slight abuse of notations, we write $t/\lambda$ for the vector $(t_1/\lambda,\ldots, t_d/\lambda)$. Moreover, \eqref{def deconvolution kernel} depends explicitly on the density $\eta$ of the noise which is supposed to be known. In practice, this knowledge could be avoided using repeated measurements (see for instance \cite{delaiglehallmeister}). 

Moreover, let $\mathcal{C}:=\left\{\bc=(c_1,\dots,c_k)\in\bR^{dk}\::\:c_j\in\cB(0,1),\:j=1,\dots,k\right\} $  be the set of possible centers in the unit ball $ \cB(0,1) $ of the Euclidean space $\R^d$.
We then introduce the following collection of noisy $k$-means minimizers:
\beqn
\label{noisykmeans}
\hat{\bc}_\lambda:=\arg\min_{\mathbf{c}\in\mathcal{C}} R_n^\lambda(\bc),~~\lambda>0,
\eeqn
where $R_n^\lambda(\bc)$ is called the deconvolution empirical risk. This quantity is defined as:
\begin{equation}\label{def empirical risk}
 R_n^\lambda(\bc)=\int_{\cB(0,1)}\gamma(\bc,x)\hat{f}_\lambda(x)dx=\frac{1}{n}\sum_{i=1}^n\gamma_\lambda(\bc,Z_i),
\end{equation}
where $\gamma_\lambda(\mathbf{c},Z)$ is the following, convolution product:
\beqnn
\gamma_\lambda(\mathbf{c},Z):=\big[\cK_\lambda\ast(\gamma(\mathbf{c},\cdot)\1_{\cB(0,1)}(\cdot))\big](Z)=\int_{\cB(0,1)}\noyaudeconvolution{\lambda}{{Z-x}}\gamma(\mathbf{c},x) dx,~~\mathbf{c}= (c_1, \ldots, c_k)\in\cC.
\eeqnn
where $ \1 $ denotes the indicator function.
Note that the restriction to the closed unit ball $ \cB(0,1) $ appears only for technicalities, and using any compact set is possible. 

Parameter $ \lambda $ in \eqref{noisykmeans} is of great interest in this paper. In particular, an appropriate choice of the bandwidth allows us to get fast rates (Section \ref{sec:nonadaptive}) and adaptive results (Section \ref{sec:adaptive}). In minimax nonparametric estimation, the standard choice of $ \lambda $ trades off a {\it bias-variance decomposition}, that is an upper bound of the measurement error (see \cite{Tsybakov08} for an overview). Thanks to \cite{Loustau12bis}, we can expect the same kind of upper bounds for the excess risk as follows:
\begin{eqnarray}
\label{bvdecomposition}
R(\hat{\mathbf{c}}_\lambda,\mathbf{c}^*)\leq (R-R_n^\lambda)(\hat{\mathbf{c}}_\lambda,\mathbf{c}^*)&\leq&(R-R^\lambda)(\hat{\mathbf{c}}_\lambda,\mathbf{c}^*)+(R^\lambda-R_n^\lambda)(\hat{\mathbf{c}}_\lambda,\mathbf{c}^*)\nonumber\\
&=:&\mbox{bias}(\lambda)+\mbox{var}(\lambda),
\end{eqnarray}
where in the sequel, for any fixed $\bc,\bc'\in\mathcal{C}$, $R^\lambda(\bc,\bc'):=\E \big[R_n^\lambda(\bc)-R_n^\lambda(\bc')\big]$ and $ \bE $ is the expectation w.r.t. $ \probaZ^{\otimes n} $. 
The first part of the decomposition is called a {\it bias} term. It depends on the unknown smoothness $ s>0 $ of the density $ f $ and on the deconvolution kernel (see Proposition \ref{prop bias} for details). The second term of this decomposition is called the {\it variance} term. It is the stochastic error of the empirical risk minimization. It depends on a standard complexity parameter and on the noise assumption (see below). This term could be controlled using empirical process theory in the spirit of \cite{svm} (see Proposition \ref{prop deviation massart}).
Therefore, as a first step, we derive in Section \ref{sec:nonadaptive} fast rates of convergence  from an optimal bandwidth $ \bar\lambda:=\bar\lambda(s) $ minimizing the latter bias-variance trade-off (see Theorem \ref{thm:nonadaptiveup}). 
\section{Fast rates for noisy clustering}
\label{sec:nonadaptive}
In this section, we propose to give a non-adaptive excess risk bound for the noisy $k$-means procedure \eqref{noisykmeans}. This result is obtained under classical assumptions from both the statistical inverse problem literature and the area of fast rates. We recall and discuss these assumptions for completeness.
% For $ \bc\in\cC $, we consider the excess risk defined as:
% \begin{equation}\label{def excess risk}
% R_{\cC}(\mathbf{c},\mathbf{c}^*):=R_{\cC}(\mathbf{c})-R_{\cC}(\mathbf{c}^*):=\int_{\cC}\big[\gamma(\mathbf{c},x)-\gamma(\mathbf{c}^*,x)\big]f(x)dx.
% \end{equation}
% Note that if the density $f$ has compact support, we can choose $\mathcal{C}$ large enough to have $R_{\cC}(\bc)=R(\bc)$. In the sequel, we write $R(\bc)$ for simplicity.
\subsection{Main assumptions}
\label{assumptions}
First of all, as in standard deconvolution problems, the use of a deconvolution kernel requires some additional assumptions on the kernel $\mathcal{K}\in L_2(\R^d)$ in \eqref{def deconvolution kernel}.
\\

\noindent\textbf{(K1)} There exist $S=(S_1,\dots,S_d)\in\R_d^+$, $K_1>0$ such that kernel $\mathcal{K}$ satisfies
$$
\mbox{supp}\mathcal{F}[\mathcal{K}]\subset [-S,S]\mbox{ and }\sup_{t\in\R^d}|\mathcal{F}[\mathcal{K}](t)|\leq K_1,
$$
where $\mbox{supp} \,g=\{x:g(x)\not= 0\}$ and $[-S,S]=\bigotimes_{v=1}^d [-S_v,S_v]$.\\

\noindent This assumption is trivially satisfied for different standard kernels, such as the \textit{sinc} kernel. This assumption arises for technicalities in the proofs and can be relaxed using a finer algebra. Moreover, in the sequel, we consider a kernel $\mathcal{K}$ of order $m\in\N$ as follows:
\begin{itemize}
\item $\int_{\R^d} \mathcal{K}(x)dx=1$,
\item $\int_{\R^d} \mathcal{K}(x)x_v^kdx=0$, $\forall k\leq m$, $\forall v\in\{1,\ldots, d\}$,
\item $\int_{\R^d} |\mathcal{K}(x)||x_v|^{m}dx<\infty$, $\forall  v\in\{1,\ldots, d\}$.
\end{itemize}
For the construction of kernels of order $m$, the univariate case is presented in \cite{Tsybakov08}. \cite{Comte_Lacour13} have detailed the multivariate case in an anisotropic framework, where the kernel can have a different order in each direction. The construction of kernels of order $m$ satisfying \textbf{(K1)} could be managed using for instance the so-called Meyer wavelet (see \cite{mallat}).

Moreover, we need an additional assumption on the regularity of the density $f$ to control the bias term. In this paper, this regularity in expressed in terms of isotropic H\"older spaces.

% Our approach requires the estimation of the underlying density $f$ via kernel estimators. These estimators design for functions well approximated by a kernel, this is in particular the case
% for functional, Hölder spaces, which we define as:
\begin{definition}
\label{def_holder} Fix $s>0$ and $L>0$, and  let
$\lfloor s\rfloor$ be the largest integer  strictly less than
$s$. The \emph{isotropic Hölder class} $\Sigma_d(s,L)$ is
the set of  functions $f:\bR^d\rightarrow\bR$ having on $\bR^d$
all partial derivatives of order $\lfloor s\rfloor$ and such that
for any $ x,y\in \bR^d$:

\begin{align*}
\left|\frac{\partial^{|p|}f(x)}{\partial
x_1^{p_1}\cdots\partial x_d^{p_d}}-\frac{\partial^{|
p|}f(y)}{\partial y_1^{p_1}\cdots\partial y_d^{p_d}}\right|&\leq L
\sum_{v=1}^d|x_v-y_v|^{s-\lfloor s\rfloor},\quad\forall\;
p\in\bN^d~:~|p|:=p_1+\cdots+p_d=\lfloor s\rfloor;\\
\sum_{m=0}^{\lfloor s\rfloor}\sum_{|
p|=m}\sup_{x\in\bR^d}\left|\frac{\partial^{|p|}f(x)}{\partial
x_1^{p_1}\cdots\partial x_d^{p_d}}\right|&\leq L,
\end{align*}
where $x_v$ and $y_v$ are the $v^{th}$ components of $x$ and $y$.

\end{definition}
In the sequel, we assume that the multivariate density $f$ of the law $P_X$ belongs to the isotropic H\"older class $\Sigma_d(s,L)$, for some $s,L>0$. It means that the density $f$ has a similar regularity in any direction. An extension to the anisotropic H\"older class is given in \cite{Loustau12bis}, which states fast rates in this case. As in standard density estimation or deconvolution, the bandwidth choice is more nasty and depends explicitly on the direction  (see also \cite{Comte_Lacour13}). It is out of the scope of the present paper.

We also need an assumption on the noise distribution $\eta$ as follows:
\\

\noindent\textbf{Noise Assumption} {\bf NA}$(\constantFourierPolynomial,\beta)$. There exists some vector $ \beta=(\beta_1,\dots,\beta_d)\in(0,\infty)^d $ and some positive constant $ \constantFourierPolynomial$ such that $ \forall t\in\bR^d $:
$$
\left|\cF[\eta](t)\right|\geq\constantFourierPolynomial\prod_{v=1}^d\left(\frac{t_v^2+1}{2}\right)^{\text- \beta_v/2}.
$$
\textbf{NA($\rho,\beta$)} deals with a lower bound on the behavior of the characteristic function of the noise density $\eta$. This lower bound is a sufficient condition to get excess risk bounds. However, to study the optimality in the minimax sense, we need an upper bound of the same order for the characteristic function. This is not the purpose of this paper. Moreover, this noise assumption is related with a polynomial behavior of the Fourier transform of $\eta$. This case is called the mildly ill-posed case in the deconvolution or statistical inverse problem  literature (see \cite{meister}). The severely ill-posed case corresponds to an exponential decreasing of the characteristic function in \textbf{NA($\rho,\beta$)}, such as a gaussian measurement error. This case is not considered in this paper for simplicity (see \cite{Comte_Lacour13} in multivariate deconvolution).

Finally, to reach fast rates of convergence, we need to introduce a margin assumption. This type of assumption is now standard in classification since the work of Tsybakov (\cite{mammen} or \cite{tsybakov2004}). In clustering, we use the following version of the well-known margin assumption (see \cite{empimini} for a related point of view): 
\\

\noindent\textbf{Margin Assumption} {\bf MA}$(\constantmargin)$: For any $ \bc\in\cC$, there exists some positive constant $ \constantmargin $ such that:
$$
\nno\bc-\bc^*(\bc)\nno^2\leq \constantmargin\big[R(\bc)-\inf_{\bc'\in\bR^{dk}}R(\bc')\big],
$$
where $\bc^*(\bc)\in\mathcal{M}$ is the nearest optimal cluster associated to $\bc$ and $\nno\cdot\nno$ stands for the Euclidean norm in $\R^{dk}$.\\

The margin assumption proposes a control of the Euclidean norm by the excess risk. Since we restrict the study to a compact set, it is easy to see that \textbf{MA($\constantmargin$)} allows us to write:
$$
\no \gamma(\bc,X)-\gamma(\bc^*(\bc),X)\no^2_{L_2(P_X)}\leq C_1\nno\bc-\bc^*(\bc)\nno^2\leq C_1 \constantmargin\E_{P_X}\big[\gamma(\bc,X)-\gamma(\bc^*(\bc),X)\big].
$$
As a result, we can use a localization principle and reach fast rates of convergence. 

The introduction of a margin assumption in clustering is actually not a novelty. It is strongly related with some well-known regularity assumptions involved in the study of the consistency of the $k$-means procedure (see \cite{pollard82}, \cite{gg}). Indeed, as shown in \cite{gg}, \textbf{MA($\kappa$)} is satisfied if $f$ is continuous and the Hessian matrix of the mapping $ \mathbf{c} \longmapsto R(\mathbf{c}) $ is positive definite at any point $\bc^*\in\mathcal{M}$. In this case, the constant $\kappa$ is related with the smallest eigenvalue of the Hessian matrix. These conditions have been introduced by Pollard to get limit theorems for the $k$-means.

Finally, \cite{levrard} has interpreted Pollard's regularity assumption in terms of well-separated classes as follows. For any $ \bc=(c_1,\ldots, c_k)\in\cC $, we associate to each center $c_i$, $i=1,\ldots ,k$ the Vorono\"i cell $V_i(\bc)$ defined as:
$$
V_i(\bc)=\{x\in\R^{d}:\min_{j=1,\ldots,k}\no x-c_j\no=\no x-c_i\no\}.
$$
Let $\partial V_i(\bc)$ be the boundary of the Vorono\"i cell $V_i(\bc)$ associated with $c_i$, for $i=1, \ldots, k$. Then, a sufficient condition to have a continuous density $f$ and a positive definite Hessian matrix is to control the sup-norm of $f$ on the union of all possible $|\mathcal{M}|$ boundaries $\partial V(\bc^*)=\cup_{i=1}^k\partial V_i(\bc^*)$, associated with $c^*\in\mathcal{M}$ as follows:
$$
\|f_{|\cup_{\bc^*\in\mathcal{M}} \partial V(\bc^*)}\|_\infty\leq \constantAtwo\inf_{\bc^*\in\mathcal{M},i=1,\ldots k}\probaX(V_i(\bc^*)),
$$
where $\constantAtwo$ is a constant depending on the dimension $d$. As a result,  the margin assumption \textbf{MA($\kappa$)} is guaranteed when the source distribution $\probaX$ is well concentrated around its optimal clusters, which is related to well-separated classes.

\subsection{A first excess risk bound}
We now present an excess risk bound for the collection of estimators introduced in \eqref{noisykmeans}, under the previous assumptions.
\begin{theorem}
\label{thm:nonadaptiveup}
Assume that {\bf NA}$(\constantFourierPolynomial,\beta)$ and {\bf MA}$(\constantmargin)$ are satisfied for some $ \beta\in(1/2,\infty)^d $, $\constantFourierPolynomial,\constantmargin>0$. Suppose $\constantsupconvolution:=\no \eta\no_\infty<\infty $ and $ f\in\Sigma_d(s,L) $ with $ s,L>0 $. Then, denoting by $\hat{\mathbf{c}}^{\bar\lambda}_n$ a solution of \eqref{noisykmeans} with:
$$ \overline{\lambda}= n^{-{1}/(2s+2\bar{\beta})},
$$ 
there exists a universal constant $ \constminimax $ depending on $ \constantmargin,w, L, d,s, \beta, \constantFourierPolynomial, k, \constantsupconvolution$ and $ |\cM|  $, and an integer $n_0\in\N^*$ such that for any $\mathbf{c}^*\in\mathcal{M}$ and any $ n\geq n_0 $:
                  
                  \[
                  \mathbb{E} R(\hat{\mathbf{c}}_{\overline{\lambda}},\mathbf{c}^*)\leq \constminimax n^{-s/(s+\bar{\beta})},
                  \]
                  where $\bar{\beta}=\sum_{v=1}^d \beta_v$.
\end{theorem}
The proof is an application of a localization approach in the spirit \cite{Massart07}, applied to the noisy set-up. As in \cite{Loustau12}, the decomposition \eqref{bvdecomposition} allows us to control the excess risk. More precisely, the variance can be controlled by mixing empirical process argues as in \cite{svm}, gathering with the noise assumption \textbf{NA($\rho,\beta$)}. The bias term is bounded using both the smoothness of $f$ and the margin assumption \textbf{MA($\kappa$)}.

Theorem \ref{thm:nonadaptiveup} improves the previous result of \cite{Loustau12bis} in the particular case of finite dimensional clustering, where a $\sqrt{\log\log(n)}$ term appears in the RHS (see Theorem 3 in \cite{Loustau12bis}). Rates of convergence of Theorem \ref{thm:nonadaptiveup} are fast rates when $\bar{\beta}<s$. It generalizes the result of \cite{levrard} to the errors-in-variables case since we can see coarsely that rates to the order $\mathcal{O}(1/n)$ are reached when $\epsilon=0$. Here, the prize to pay for the inverse problem is the quantity $\sum_{i=1}^d \beta_i$, related to the tail behavior of the characteristic function of the noise distribution $\eta$ in {\bf NA}$(\constantFourierPolynomial,\beta)$. 

% Lower bounds of the form $\mathcal{O}(1/\sqrt{n})$ have been stated in the direct case by \cite{lbclustering} for general distribution. 
An open problem is to derive the optimality of Theorem \ref{thm:nonadaptiveup} in the minimax sense, under the margin assumption \textbf{MA($\constantmargin$)} and the noise assumption {\bf NA}$(\constantFourierPolynomial,\beta)$. 
% For this purpose, we need to construct configurations where both the margin assumption \textbf{MA($\constantmargin$)} and the noise assumption {\bf NA}$(\constantFourierPolynomial,\beta)$ could be used in a careful way. 
In this direction, \cite{pinkfloyds} proposes a complete minimax study in classification with error-in-variables by using a comparable estimation procedure. We then conjecture that the rate $ n^{-s/(s+\bar{\beta})} $ is minimax over Hölder spaces.

\section{Bandwidth Selection}
\label{sec:adaptive}
In this section, we turn out into the main issue of this paper: the data-driven choice of the bandwidth $\lambda>0$ in the collection of estimators $\{\hat\bc_\lambda,\lambda>0\}$ defined in \eqref{noisykmeans}. The goal is to reach adaptive excess risk bound similar to Theorem \ref{thm:nonadaptiveup} for a choice of $\lambda$ which does not depend on the smoothness of $f$.\\

%Cross-validation methods are the most popular tools for choosing a bandwidth or any other regularization parameter in statistics. The idea is to choose the nuisance parameter that minimizes an estimation of the considered risk. 
In supervised learning (such as regression or binary classification), it is standard to choose a bandwidth - or a tuning - parameter using a decomposition of the set of observations. A training set is used to construct a family of candidate estimators, each one associated with a different value of the bandwidth.  Then, a test set allows to estimate the generalization performances of each candidate. It gives rise to the family of cross-validation methods, or aggregation procedures. Unfortunately, in unsupervised tasks, this simple estimation is not possible. The lack of efficiency of cross-validation methods in clustering has been illustrated in \cite{tib} for the problem of choosing $k$ in the $k$-means.
%such as the leave-one-out cross validation. Unfortunately, in unsupervised model, this simple estimation is not possible. The problem of cross-validation methods in clustering has been illustrated in \cite{tib} for the problem of choosing $k$ in the $k$-means. In this case, using cross-validation, the choice of $k$ is dramatically too big, since growing $k$ leads to observations close to one of the numerous center. 
Moreover, in the presence of errors in variables, such as in deconvolution, it is also quiet standard to perform cross-validation to choose the bandwidth of a deconvolution estimator. As described in \cite{meister}, it is possible to estimate the squared risk $\no \hat{f}_\lambda -f\no^2$ with Plancherel theorem, leading to the estimation of the Fourier transform of the unknown density. 
%Leave-one-out techniques are introduced and lead to a data-driven selection. 
However, in our framework, this method seems hopeless since the optimal value of $\lambda$ does not minimize a squared risk but an excess risk of the form \eqref{def excess risk}.
%Aggregation methods, introduced by Nemirovski (\cite{nemirovski}), could be used to select a nuisance parameter, such as the bandwidth of a kernel estimator (see \cite{rigoletsybakov}) or any other regularization parameters (see for example \cite{loustau1} or \cite{lecue} in binary classification). For this problem, an aggregation method is based on the same philosophy as a cross-validation one. A training set is used to construct a family of basic estimators, for different values of the nuisance parameter. Then, a linear combination of these basic estimators is proposed thanks to a test set. This linear (or convex) combination is based on weights, calculated thanks to the estimation of the performances of each basic estimator on a test set. In this sense, the principle is similar to cross-validation and is not adapted to the noisy clustering, where it seems to be untractable to evaluate each basic learner thanks to a noisy test set.
Eventually, model selection was introduced for selecting the hypothesis space over a sequence of nested models (e.g. finite dimension models) with a fixed empirical risk. 
%The choice of $k$ in the $k$-means can be seen as a model selection problem, where $k$ is associated to the finite dimensional euclidean space $\R^{dk}$. \cite{fisherbiau} proposes to penalize the empirical distortion to avoid the choice of a large values of $k$. Moreover,
Penalization methods are also suitable to choose smoothing parameters of well-known statistical methods such as splines, SVM or Tikhonov regularization methods. The idea is to replace the choice of the smoothing parameter by the choice of the radius into a suitable ellipsoid. Unfortunately, here,  the nuisance parameter $ \lambda $ affects directly the empirical risk \eqref{def empirical risk}, and a model selection method can not be directly applied in this context.\\

Theorem \ref{thm:nonadaptiveup} below motivates the use of a comparison method based on the Lepski's principle (\cite{Lepski90}). Indeed, the non-adaptive choice of $\bar\lambda=n^{\text-1/(2s+2\bar{\beta})}$ trades off a bias-variance decomposition  of the excess risk \eqref{bvdecomposition} and allows to get fast rates of convergence. As a result, the Lepski's principle appears as the most commonly tool to propose an adaptive estimator $\hat{\bc}_{\hat\lambda}$, where $\hat\lambda$ mimics the oracle $\bar\lambda$ of Theorem \ref{thm:nonadaptiveup}. The built estimator $\hat{\mathbf{c}}_{\hat\lambda}$ will be called adaptive since it does not depend on the smoothness $s$. 

To define the selection rule, we first remind some definitions and notations. Given a kernel $ \mathcal{K} $ satisfying the previous assumptions, we note $ \|\mathcal{K}\|_1 $ the $ L_1 $-norm of the kernel on $ \bR^d $. The constant $ \constantsupconvolution:=\no \eta\no_\infty $ is the sup-norm of the noise density $ \eta $, whereas $ \rho>0 $ and $ \bar{\beta}=\sum_{v=1}^d \beta_v $ are parameters involved in the noise assumption {\bf NA}$(\constantFourierPolynomial,\beta)$. Moreover, $ \kappa $ is the  constant in the margin assumption {\bf MA}$(\constantmargin)$. In the sequel, $\mathcal{V}(d)=\pi^{d/2}/\Gamma(d/2+1)$, where $\Gamma(\cdot)$ stands for the Gamma function.\\ Define the threshold term:
\beqn
\label{threshold}
\threshold{\lambda}:=\frac{2^{10}\sqrt{2}\mathcal{V}(d)\no \cK\no_1^2\kappa\eta_\infty}{\rho^2}\frac{\lambda^{\text-2\bar\beta}\log(n)}{n},
\eeqn
where $ \lambda $ belongs to the bandwidth set $ \Lambda:=[\lambda_{\min},\lambda_{\max}] $ with 
$$
\lambda_{\min}:=\frac{\log^{1/\bar\beta}(n)}{n^{1/2\bar\beta}}~\text{and}~ \lambda_{\max}:=\big(1/\log(n)\big)^{1/(2s^{+}+2\bar\beta)},
$$
where $s^+>0$ is an upper bound on the regularity index of $f$. In this section, we take $n$ sufficiently large such that $ n^{-1/(2s+2\bar\beta)}\in\Lambda $. Moreover, for some constant $a\in(0,1)$, we set:
 $$ \Lambda_a:=\left\{\lambda\in\Lambda:\:\exists
  m\in\bN\:,\:\lambda=\lambda_{\max}a^m\right\},$$
   a discrete exponential net on the bandwidth set with cardinal $ |\Lambda_a| $.
  
  We are ready to introduce the adaptive bandwidth choice, called ERC (Empirical Risk Comparison):
\beqn
\label{def:erc}
\hat\lambda=\max\left\{\lambda\in\Lambda_a\::\:R_n^{\lambda'}(\hat\bc_{\lambda})-R_n^{\lambda'}(\hat\bc_{\lambda'})\leq 3\threshold{\lambda'},\:\forall {\lambda'}\leq{\lambda}\right\}.
\eeqn
The noisy $k$-means estimator \eqref{noisykmeans} with bandwidth $\hat\lambda$ chosen from ERC rule \eqref{def:erc} has the following property.
\begin{theorem}\label{th adaptivity lambda}
Assume that {\bf NA}$(\constantFourierPolynomial,\beta)$ and {\bf MA}$(\constantmargin)$ are satisfied for some $ \beta\in(1/2,\infty)^d $, $\constantFourierPolynomial,\constantmargin>0$. Suppose $\constantsupconvolution:=\no \eta\no_\infty<\infty $ and $f\in\Sigma(s,L)$, where $s\in [0,s^+)$ and $ L>0 $. Then, there exists a universal constant $ \constadaptative$ depending on $ \constantmargin,w, L, d,s, \beta, \constantFourierPolynomial, k, \constantsupconvolution,$ $ |\cM|  $, and $ n_1\in\bN $ such that for any $\mathbf{c}^*\in\mathcal{M}$ and any $ n\geq n_1 $, estimator $\hat{\mathbf{c}}_{\hat\lambda}$  with $ \hat \lambda$ selected by ERC \eqref{def:erc} satisfies:
                  
                  \[
                  \mathbb{E} R(\hat{\mathbf{c}}_{\hat\lambda},\mathbf{c}^*)\leq \constadaptative \left(\frac{\log(n)}{n}\right)^{s/(s+\bar\beta)},
                  \]
                  where $\bar{\beta}=\sum_{v=1}^d \beta_v$.
\end{theorem}
Theorem \ref{th adaptivity lambda} is an adaptive upper bound for the estimator $\hat{\bc}_{\hat{\lambda}}$, where  $\hat{\lambda}$ is chosen from the ERC selection rule \eqref{def:erc}. The estimator $\hat{\bc}_{\hat{\lambda}}$ is then adaptive w.r.t. the smoothness $ s $. This adaptive excess risk bound coincides with the non-adaptive previous result of Theorem \ref{thm:nonadaptiveup}, up to an extra log term. This is the prize to pay for the data-driven property of the procedure. 

Let us remind that it is standard to pay a $ \log(n) $ factor  in pointwise estimation (see \cite{Lepski90} and \cite{Brow_Low96}). However, it is well-known that there is no prize to pay for adaptivity in global estimation (e.g. $ L_p $-norm). In the problem of noisy clustering, the choice of $\lambda$ concerns the estimation of the density $f$. This estimation is used in the procedure of noisy $k$-means, where we plug $\hat{f}_\lambda$ into the true risk. At the first glance, we could conjecture that a global estimation of $f$ is sufficient. 
%However, if we take a closer look at the problem of clustering, or more generally classification, this basic intuition seems to be wrong. Indeed, \cite{loustau12} illustrates a suboptimal phenomenon for the problem of classification with corrupted observations in the following case. If we choose to use a projection estimator of the density $f$ of the direct inputs $X_i$'s, a similar bias variance decomposition of the excess risk is still available. In this case, the dimension of the projection estimator has to be chosen to reach the bias variance trade off. Interestingly, in this case, optimal fast rates of convergence could not be attained using the same algebra as for deconvolution estimator. the main reason is that in the control of the bias term, we need to upper bounds optimally the local bias term $\E\hat{f}(x_0)-f(x_0)$, at any point $x_0$. As a result, a projection estimator is not optimal in this case. This phenomenon clearly highlight the nature of the estimation problem we have at hand: a 
However, a closer look at the problem of noisy clustering, or more generally noisy classification, clearly highlights the nature of the estimation problem we have at hand: a pointwise problem of estimation of a density $f$ thanks to noisy or corrupted measurements (see \cite{Loustau12}, Theorem 3). As a result, we can conjecture that the result of Theorem \ref{th adaptivity lambda} is optimal, in the adaptive minimax sense, as in standard pointwise estimation. 

The threshold term $ \delta_\lambda $ - which comes from the control of the stochastic part of the excess risk, see Proposition \ref{prop deviation massart} - has the following form (see \eqref{threshold}):
$$
\delta_\lambda=C_{\mathrm{adapt}}\frac{ \lambda^{\text-2\bar\beta}\log(n)}{n},
$$
where the (large) constant $C_{\mathrm{adapt}}>0$  depends on the margin constant $ \constantmargin $ which is possibly unknown. Indeed, by definition, it depends on the underlying density $f$. In practice, we recommend a painstaking calibration of this constant. From the theoretical point of view, this constant could be chosen from the {\it propagation} method suggested by \cite{Spokoiny_Vial09}.

The proof of Theorem \ref{th adaptivity lambda}, given in Section \ref{sec:proofs}, is based on the standard  Lepski's principle  subject to
major modifications. This rule traditionally uses a comparison between estimators (possibly defined as minimizers) indexed by a nuisance parameter to get the  best one from the considered family. In our context, we have proposed to compare empirical risks indexed by the nuisance parameter (bandwidth) to control the performance of the selected risk minimizer in terms of excess risk. Note that this idea was suggested by \cite{Polzehl_Spokoiny06} via local likelihood Comparison with the Kullback divergence as measurement error (see next section for more details). To the best of our knowledge, it is the first time that the Lepski's principle is applied in statistical learning on this way. 

In the following section, we give sufficient conditions to apply the ERC method in the more general context of $M$-estimation depending on a nuisance parameter. Many examples could be considered in future works.

\section{ERC's extension}\label{sec:gen}
In this section, we propose an extension of the ERC selection rule \eqref{def:erc} to a more general context of $M$-estimation that we illustrate by some examples. To this end, let us introduce the random variable $\mathcal{Z}$ on some probability space $(\Omega,\mathcal{F},\P)$, with law $\P_{\mathcal{Z}}$ defined on $\mathcal{X}$. Given a sample of i.i.d. random variables $ \cZ_1,\dots,\cZ_n\in\mathcal{X}$ with probability law $\P_{\mathcal{Z}}$, we construct an empirical risk denoted by $ \cR^\lambda_n(\cdot) $, where $ \lambda>0 $ is a bandwidth to choose via ERC method. This parameter can be the bandwidth of a standard kernel in a localization approach (see Examples 4. and 5.), or the bandwidth of a deconvolution kernel in error-in-variables models (see Examples 1., 2. and 3.). We consider the collection of $M$-estimators:
\begin{equation}
\label{genMest}
 \hat{g}_n^\lambda\in\argmin_{g\in\cG}\cR^\lambda_n(g),~~\lambda>0,
\end{equation}
where $ \cG $ is a fixed family of candidates that depends on the considered problem. Given the collection of estimators \eqref{genMest}, we focus on the selection of the nuisance parameter $\lambda>0$ appearing in the empirical risk. Without loss of generality, we assume that $ \lambda$ belongs to the set $ \Delta_n= [\lambda_-,\lambda^+] $, where  $\lambda_-\leq\lambda^+$ and $ \lambda_-,\lambda^+\to0 $ as $ n\to\infty $.

The risk of the estimators \eqref{genMest} is measured thanks to a quantity $\mathcal{R}( \hat{g}_n^\lambda)$, called the true risk. To obtain good properties for $\hat{g}_n^\lambda$, the first condition on $\mathcal{R}_n^\lambda(\cdot)$ is that it has to be an asymptotically unbiased estimator of the true risk $\mathcal{R}(\cdot)$. To this end, for any fixed estimator $g\in\mathcal{G}$,  we introduce the expectation of the empirical risk as $ \cR^\lambda(g)=\bE \cR^\lambda_n(g), $
where $ \bE $ denotes the expectation w.r.t. the product probability of the training set $(\mathcal{Z}_1,\ldots, \mathcal{Z}_n)$. 
Then, the empirical risk has the following property:
$$
\lim_{n\to\infty}\bE\cR^\lambda_n(g)=\mathcal{R}(g),\mbox{ for any fixed }g\in\mathcal{G}.
$$
We also denote as $g^*$ the best possible decision rule in $\mathcal{G}$, called an oracle, defined as:
$$
g^*\in\argmin_{g\in\cG}\cR(g).
$$
Eventually, the performance of the estimate $ \hat{g}_n^{\lambda} $ is measured via the excess risk
$ \cR(\hat{g}_n^{\lambda},g^*):= \cR(\hat{g}_n^{\hat \lambda})- \cR(g^*) $.
\\

As illustrated in Section \ref{sec:nonadaptive} for the particular case of noisy clustering, the performances of $\hat{g}_n^\lambda$ in terms of excess risk is strongly related with the behaviour of $\mathcal{R}_n^\lambda(g)$ as an estimator of the true risk $\mathcal{R}(g)$. The message of this section is the following : if the behaviour of $\cR^\lambda_n(g)$ in \eqref{genMest} can be decomposed into a standard bias-variance decomposition, an adaptive choice $\hat \lambda$ can be proposed using the ERC rule. It leads to the following assumptions that are sufficient conditions to propose an optimal adaptive procedure:
\\

\noindent\textbf{Bias/Variance Conditions.} \\

{\it 
{\it 1.}
There exists an increasing function denoted by $ \textbf{Bias}(\cdot) $ such that:
$$
\big|(\cR^\lambda-\cR)(g,g^*)\big|\leq \textbf{Bias}(\lambda) +\frac 1 4 \cR(g,g^*),~~\text{for all}~ g\in\cG~\text{and}~\lambda\in\Delta.
$$

{\it 2.} There exists a decreasing function denoted by $ \textbf{Var}_t(\cdot) $ ($ t\geq0 $) such that:
$$
\bP\left(\sup_{g\in\cG}\left\{\big|(\cR^\lambda_n-\cR^\lambda)(g,g^*)\big|-\frac 1 4 \cR(g,g^*)\right\}>\textbf{Var}_t(\lambda)\right)\leq e^{\text- t},~~\text{for all}~ \lambda\in\Delta ~\text{and}~ t\geq0.
$$}
\\

\noindent The above conditions ensure a control of the excess risk using the bias-variance decomposition \eqref{bvdecomposition}. It allows to find a quantity $\lambda^*$ which trades off both terms $\textbf{Bias}(\lambda)$ and $\textbf{Var}_t(\lambda)$. In particular, using \eqref{bvdecomposition},  Bias/Variance conditions {\it 1.} and {\it 2.} lead to the following exponential inequality:
$$
\P\Big(\mathcal{R}(\hat g_n^\lambda)-\mathcal{R}(g^*)\geq 2\textbf{Bias}(\lambda)+2\textbf{Var}_t(\lambda)\Big)\leq e^{\text- t},~~\text{for all}~ t\geq0.
$$

A close inspection of the proof of
Theorem \ref{th adaptivity lambda} highlights that Bias/Variance Conditions are sufficient  to show the adaptive property of the ERC rule in this general context as follows. Let $ \Delta_a $ be a discrete exponential net on the set $ \Delta $. Then, the general ERC rule is given by:
\beqn
\label{def general erc}
\hat\lambda=\max\left\{\lambda\in\Delta_a\::\:\cR_n^{\lambda'}(\hat g^{\lambda}_n)-\cR_n^{\lambda'}(\hat g^{\lambda'}_n)\leq 8\textbf{Var}_t(\lambda),\:\forall {\lambda'}\leq{\lambda}\right\}.
\eeqn
Thus, using the latter Bias/Variance conditions, one can establish
oracle inequatilies for the selection rule \eqref{def general erc} as follows:
\begin{theorem}
\label{gentheo}
Suppose the Bias/Variance Conditions 1. and 2. hold. Then, there exists a universal constant $ C_3 $ such that
$$
\mathbb{E} \cR(\hat g^{\hat\lambda}_n,g^*)\leq C_3\left(\inf_{\lambda\in\Delta}\Big\{\textbf{Bias}(\lambda)+\textbf{Var}_t(\lambda)\Big\}+e^{\text- t}\right),~~\text{for all}~ t\geq0,
$$
where $\hat\lambda$ is chosen in \eqref{def general erc}.
\end{theorem}
We omit the proof of Theorem \ref{gentheo} since it can be deduced coarselly from the proof of Theorem \ref{th adaptivity lambda}. This result allows to get a control of the expected excess risk of $\hat{g}_n^{\hat\lambda}$ via a data-driven parameter $\hat\lambda$. The adaptive estimator performs as well as the best one, i.e. the one minimizing the Bias/Variance trade-off in the family $ \{\hat g^{\lambda}_n,~{\lambda\in\Lambda}\} $. This result could be of great interest in  many statistical learning context. We now give some examples in the context of noisy data, and then we turn out into the problem of local $M$-estimation.

\paragraph{Example 1: Noisy Clustering.} The framework of Section \ref{sec:nonadaptive} exactly falls into the general model of this section. Indeed, in the problem of clustering with noisy inputs \eqref{noisydata}, the empirical risk is defined as in \eqref{def empirical risk}:
$$
\cR^\lambda_n(\bc)=\sino\int\noyaudeconvolution{\lambda}{{Z_i-x}}\min_{j=1,\dots,k}\|x-c_j\|_2^2 dx
$$
where $ \bc=(c_1,\dots,c_k)\in\bR^{dk} $ and $\lambda>0$ is the bandwidth of the deconvolution estimator. Conditions {\it 1.} and {\it 2.} correspond to Propositions \ref{prop bias} and \ref{prop deviation massart} (see Section \ref{sec:proofs}). 
%More precisely, under the noise assumption, gathering with the Margin assumption, the variance term satisfies  $\textbf{Var}_t(\lambda)\sim\frac{(1+t)}{n\lambda^{\bar\beta}}$ (see Proposition \ref{prop deviation massart}). Moreover, if $f\in\Sigma_d(\gamma,L)$, and under the kernel assumptions,  Condition {\it 1.} is satisfied with $\textbf{Bias}(\lambda)\sim \lambda^{2\gamma} $ (see Proposition \ref{prop bias}). It leads to the adaptive excess risk bound of Theorem \ref{th adaptivity lambda}. \\
\paragraph{Example 2: Discriminant Analysis (\cite{pinkfloyds}).}
The model of discriminant analysis is the supervised counterpart of the clustering problem. Suppose we have at hand two samples $Z_1^1,\ldots, Z_n^1$ and $Z^2_1,\ldots, Z^2_m$ of the form \eqref{noisydata}, with density $p*\eta$ and $q*\eta$. The aim is to design a decision set $G\subset\R^d$ such that if $x\in G$, $x$ is associated to the density $p$, and $q$ otherwise. In this context, \cite{pinkfloyds} states minimax fast rates of convergence for the excess risk of classsification $\mathcal{R}(G)-\mathcal{R}(G^*)$, where $\mathcal{R}(G)$ is defined as:
$$
\mathcal{R}(G)=\frac{1}{2}\left[\int_{G^C} p(x)dx+\int_G q(x)dx\right],
$$
and $ G^*\in\argmin_{G}\mathcal{R}(G) $.
Using a deconvolution kernel of the form \eqref{def deconvolution kernel}, we can introduce an asymptotically unbiased estimator of $\mathcal{R}(G)$ given by:
$$
\mathcal{R}_n^\lambda(G)=\frac{1}{2}\left[\int_{G^C} \hat{p}_\lambda(x)dx+\int_G \hat{q}_\lambda(x)dx\right],
$$
where $\hat{p}_\lambda$ and $\hat{q}_\lambda$ are deconvolution kernel estimators of $p$ and $q$ with given bandwidth $\lambda>0$. We also define the estimated decision set $ \hat{G}_n^\lambda\in\argmin_{G}\mathcal{R}_n^\lambda(G) $. In this context, \cite{pinkfloyds} have proved minimax fast rates of convergence for the excess risk $\mathcal{R}(\hat{G}_n^\lambda)-\mathcal{R}(G^*)$, using a bias-variance decomposition as in \eqref{bvdecomposition}. The generalization of ERC in this particular minimax framework illustrates that a minimax adaptive excess risk bounds could be stated by using the ERC rule \eqref{def general erc}, up to an extra log term.

\paragraph{Example 3: Quantile estimation (\cite{quantileestimation}).} Given noisy data \eqref{noisydata}, the objective of quantile estimation is to estimate a $\tau$-quantile of the distribution $P_X$, which is given by:
$$
q_\tau:=\arg\min_{\eta\in\R}\int_{-\infty}^{+\infty}(x-\eta)(\tau-\ind_{x\leq\eta})f(x)dx=:\arg\min_{\eta\in\R}\cR(\eta),
$$
where $\tau\in(0,1)$. \cite{quantileestimation} proposes to estimate $q_\tau$ using a deconvolution kernel estimator of the form \eqref{def deconvolution kernel}. The proposed estimator is defined as:
$$
\hat{q}_\tau^\lambda=\arg\min_{\eta\in\R}\int_{-\infty}^{+\infty}(x-\eta)(\tau-\ind_{x\leq\eta})\hat{f}^\lambda_n(x)dx=:\arg\min_{\eta\in\R}\cR_n^\lambda(\eta),
$$
where $\hat{f}_\lambda(\cdot)$ is given in \eqref{decest} with  bandwidth $ \lambda>0$. Interestingly, they prove convergence rates for the quantity $|q_\tau-\hat{q}_\tau^\lambda|$, (provided that $f\in\Sigma(s,L)$) for an optimal bandwidth $\lambda:=\bar{\lambda}(s)$ which trades off a bias-variance decomposition similar to \eqref{bvdecomposition}. As usually, the choice is non-adaptive, and the authors use the standard Lepski's procedure to give adaptive upper bounds for the measurement error $|q_\tau-\hat{q}_\tau^\lambda|$. 

We conjecture that Theorem \ref{gentheo} could be stated from the previous mentioned work. Thus, ERC method \eqref{def general erc} could be used to obtain sharper risk bounds, i.e. adaptive upper bounds for the excess risk $\mathcal{R}(\hat{q}_\tau^{\hat\lambda})-\mathcal{R}(q_\tau)$. In particular, this result could recover \cite{quantileestimation}.
\paragraph{Example 4: Local Fitted Likelihood (\cite{Polzehl_Spokoiny06}).} Let us introduce a sample of independent random variables $\mathcal{Z}_i=(X_i,Y_i)\in[0,1]^n\times\bR^{n},$ $i=1,\ldots ,n$, where $P_{\theta_i}$ denotes the distribution of $Y_i$ with given probability density $ p(\cdot,\theta_i) $, with a parameter $\theta_i=\theta(X_i)$. In \cite{Polzehl_Spokoiny06}, the aim is to estimate the quantity $ \theta(x) $ at a given point $x$ (pointwise estimation).
% Then, $Y_i$ can be seen as the observation at location or time $X_i$. 
% This model could be viewed as a natural nonparametric extension of the parametric model where $Y_i$, $i=1, \ldots, n$ has law $P_\theta$ with constant parameter $\theta\in\Theta$. 
This model contains standard nonparametric problems such as, for example:
\begin{itemize}
 \item Gaussian regression where $Y_i=\theta(X_i)+\epsilon_i$, $i=1,\ldots, n$, with gaussian errors $\epsilon_i$;
\item Binary classification model where $\theta(x)=\P(Y_i=1|X=x)$;
\item Inhomogeneous exponential model $ p(\cdot,\theta_i)=e^{-\cdot/\theta_i}/\theta_i $;
\item Inhomogeneous Poisson $ \P(Y_i=k|X_i)=\theta^k_ie^{-k/\theta_i}/k!,~k\in\bN $.
\end{itemize}
In this problem, one usually applies the local version of the well-known likelihood method. It gives rise to the minimization of a localized, negative log-likelihood as follows:
$$
\hat{\theta}^\lambda\in\argmin_{t\in\bR}\frac{1}{n}\sum_{i=1}^n-\log \left(p(Y_i,t)\right)\frac{1}{\lambda}\mathcal{K}\left(\frac{X_i-x}{\lambda}\right),
$$
where $\mathcal{K}(\cdot)$ is a kernel function and $ \lambda>0 $ plays the role of a bandwidth. In this context, the localized likelihood depends on a bandwidth. Moreover, in such a framework, the accuracy of a given $\hat \theta^\lambda$ can be measured thanks to the Kullback-Leibler divergence between the two probability $P_\theta$ and $P_{\hat\theta^\lambda}$, given by
$$
\mathcal{K}(\theta,\hat\theta^\lambda)=\mathbb{E}_{P_\theta}\log\frac{ p(Y,\theta)}{ p(Y,\hat\theta^\lambda)}.
$$
\cite{Polzehl_Spokoiny06} proposes a data-driven selection of $ \lambda $ using Lepski's principle comparing localized, negative log-likelihoods. This rule is similar to ERC method \eqref{def general erc} where the empirical risk corresponds to the localized, negative log-likelihood, which is an asymptotically unbiased estimator of the risk $\cR(\cdot):=-\bE_{P_\theta}\log p(Y,\cdot)$. In this framework, the excess risk $ \cR(\hat{\theta}^\lambda)-\cR(\theta) $ associated to the likelihood coincides with the Kullback-Liebler divergence:
$$
\cR(\hat{\theta}^\lambda)-\cR(\theta)=\mathcal{K}(\theta,\hat\theta^\lambda).
$$
Eventually, we note that Bias/Variance Conditions can be established thanks to a margin assumption $ \mathcal{K}(\theta,\hat\theta^\lambda)\sim I^{\text-1} |\theta,\hat\theta^\lambda| $, where $ I $ denotes the Fisher information. Then, ERC method is relevant and coincides with the adaptive method introduced by \cite{Polzehl_Spokoiny06}.
\paragraph{Example 5: Robust M-estimation (\cite{chichibernoulli}).} Let us consider nonparametric regression where we observe a sequence of i.i.d. random variables $ \cZ_i=(X_i,Y_i)\in[0,1]^n\times\bR^n $, $i=1,\ldots ,n$ satisfying:
$$ Y_i=f(X_i)+\xi_i,
$$
 where the design $ X_i $ is uniformly distributed on $ [0,1] $ and the noise $\xi_i$ is symmetric, possibly heavy-tailed and such that $ \bP(\text-1\leq\xi_i\leq1)>0 $. We are interested into the pointwise estimation of the regression function $ f:[0,1]\to\bR $ at a given $ x_0 $. In such a framework, localized approaches (e.g. with a kernel function) are usually used. In particular, we can focus on the local Huber estimator as follows:
$$
\hat f_\lambda(x_0):=\argmin_{t\in\bR}\sino \rho_{H}(Y_i-t)\frac{1}{\lambda}\mathcal{K}\left(\frac{X_i-x_0}{\lambda}\right),
$$
where $ \rho_{H}(u)=(u^2/2)\1_{|u|\leq1}+(|u|-1/2)\1_{|u|>1} $ is the Huber contrast (cf. \cite{Huber64}), $\mathcal{K}(\cdot)$ is a kernel function and $ \lambda>0 $ is the so-called bandwidth. The estimator was recently investigated by \cite{chichibernoulli} for its robustness properties. 

Following the approach of Theorem \ref{gentheo}, we can use the general ERC rule to get adaptive upper bounds for the mean square error of $ \hat f_{\hat\lambda}(x_0) $. Indeed, as a first step, since the Huber contrast is smooth enough, a margin assumption holds:
$$
\cR(t,f(x_0))\sim \frac{|t-f(x_0)|}{\bE\rho''(\xi)}, 
$$ 
for $ t $ closed to $ f(x_0) $. Then, from a careful look at the proof of  \cite{chichibernoulli}, there is nice hope that the Bias/Variance Conditions are satisfied. Eventually, a direct application of Theorem \ref{gentheo} is possible.
\section{Conclusion}
\label{sec:conclusion}
This paper could be seen as a first step into the study of adaptive noisy clustering. Several problems remain open and could be the core of future works.\\

Firstly, we obtain in Theorem \ref{thm:nonadaptiveup} a non-adaptive excess risk bound in clustering with noisy data for a collection of deconvolution empirical risk minimizers. This bound highlights the presence of fast rates of convergence, which improves the previous result stated in \cite{Loustau12bis} using Koltchinskii's localization approach (see \cite{kolt}). Here, fast rates of the form $\mathcal{O}(n^{-s/(\bar{\beta}+s)})$ are obtained, where $s>0$ is the H\"older regularity of the density $f$ and $\bar\beta$ deals with the asymptotic behaviour of the characteristic function of the noise. These rates are reached for a non-adaptive bandwidth choice $\bar\lambda= n^{-1/(2\bar{\beta}+2s)}$. Then, we turn out into the main issue of the paper: the adaptive choice of the bandwidth of the estimator. We introduce a new selection rule based on Lepski's heuristic, where empirical risks are compared instead of estimators. This rule, called ERC, allows us to get an adaptive excess risk bound which coincides 
with the optimal upper bound (of Theorem \ref{thm:nonadaptiveup}), up to a $\log(n)$ term. This prize to pay for the adaptivity seems to be optimal, as discussed at the end of Theorem \ref{th adaptivity lambda}. 

The introduction of the ERC rule in noisy clustering leads to several open problems. First of all, the proposed selection rule suffers from the dependence on the threshold term $\delta_\lambda$ in \eqref{threshold}, which depends on unknown constants, such as the margin constant $\kappa$ in \textbf{MA($\kappa$)}. An interesting but challenging open problem is to investigate the adaptivity with respect to the margin assumption. Moreover, this threshold is not realistic and a precise calibration of this term is of practical interest. In this direction, it could be interesting to develop the propagation method presented in \cite{Spokoiny_Vial09}.
Another interesting direction is to extend the result of this paper to the anisotropic case. From \cite{Loustau12bis}, we know the presence of fast rates of convergence for the collection of noisy $k$-means, where the density $f$ has an anisotropic H\"older regularity. These results are quiet similar to the isotropic case, except that the choice of the bandwidth is more nasty. Exactly as in a deconvolution framework, in the anisotropic case, the optimal bandwidth is not the same in each direction. As a result, the problem of adaptivity is more difficult in this case, since we have to consider a $d$-dimensional grid of parameters $\lambda_j$, $j=1, \ldots, d$. For this purpose, the application of a ERC based on \cite{Goldenshluger_Lepski11} is a challenging open problem.

 The construction of an algorithm to compute the ERC rule is also of first interest. This could be done thanks to the recent developments stated in \cite{bl12}, where a noisy $k$-means algorithm is proposed. Then, it could be interesting to test over simulated as well as real datasets the problem of choosing the bandwidth in the algorithm. An implementation of the ICI algorithm will be efficient to avoid the calculation of all the estimators in the collection of noisy $k$-means.\\

Finally, in this contribution, the Lepski's heuristic is introduced for the first time to get adaptive excess risk bounds in statistical learning. As discussed in Section \ref{sec:gen}, the data-driven choice of the bandwidth proposed in this paper can be applied to a more general context of $M$-estimation, where a nuisance parameter appears in the empirical risk. In this general context, ERC appears to be useful to get adaptive excess risk bounds, provided that a suitable bias-variance decomposition is available for the empirical risk (see Theorem \ref{gentheo}). We conjecture that the guiding thread of this paper could be use in a variety of statistical models, where parameter selection is involved. 

\section{Proofs}\label{sec:proofs}

\subsection{Basic Results}
Proofs of Theorem \ref{thm:nonadaptiveup}-\ref{th adaptivity lambda} are based on the following two basic propositions.

\begin{proposition}[Bias control]
\label{prop bias}
Suppose $f\in\Sigma_d(s,L)$. Let $\mathcal{K}(\cdot)$ a kernel of order $\lfloor s\rfloor$ with respect to $\nu$, and assume that {\bf MA}$(\constantmargin)$ is satisfied for some positive constant $\constantmargin$. For any $ \epsilon>0 $ and any $ \lambda>0 $, it holds:
\begin{equation*}
\,\left|(R-R^\lambda)(\bc,\bc^{*})\right|\leq \constbias\constantmargin\epsilon\lambda^{2s}+\frac{1}{2\epsilon}R(\bc,\bc^{*}),~~\text{for all}~ \bc\in\cC,
\end{equation*}
where $\constbias=16[\mathcal{V}(d)]^2\left( \int \big|\mathcal{K}(u)\big|\frac{Ld|u|^{s}}{l!}du\right)^2$ with $\mathcal{V}(d)=\pi^{d/2}/\Gamma(d/2+1)$ and $\Gamma(\cdot)$ stands for the Gamma function.
\end{proposition}
\begin{proof}
We consider the case $d=1$ for simplicity. Using the elementary property $\E_{P_Z} \noyaudeconvolution{\lambda}{Z-x}=1/\lambda\E_{P_X} \mathcal{K}\left((X-x)/\lambda\right)$, gathering with Fubini, we can write:
\begin{equation*}
(R^\lambda-R)(\bc,\bc^{*})
=\int_{\cB(0,1)}\big(\gamma(\bc,x)-\gamma(\bc^*,x)\big)\int_{\bR} \mathcal{K}(u)\big(f(x+\lambda u)-f(x)\big)dudx,
\end{equation*}
Now, since $f$ has $l=\lfloor s\rfloor$ derivatives and $ \mathcal{K}(\cdot) $ a kernel of order $ l $, there exists $\tau\in ]0,1[$ such that:
\beqnn
\int_{\bR} \mathcal{K}(u)\left(f(x+\lambda u)-f(x)\right)du&\leq& \int_{\bR} \mathcal{K}(u)\left(\sum_{k=1}^{l-1}\frac{f^{(k)}(x)}{k!}(\lambda u)^k+\frac{f^{(l)}(x+\tau\lambda u)}{l!}(\lambda u)^l\right)du\\
&\leq & \int_{\bR} \mathcal{K}(u)\left(\frac{(\lambda u)^l}{l!}(f^{(l)}(x+\tau\lambda u)-f^{(l)}(x))\right)du\\
&\leq & \lambda^s \int_{\bR} \big|\mathcal{K}(u)\big|\frac{L|u|^{s}}{l!}du,
\eeqnn
where we use in last line the H\"older smoothness of  $f$. From \eqref{eq bound m} and {\bf MA}($\constantmargin$), we have that 
$ |\gamma(\bc,X)-\gamma(\bc^*,X)|\leq 4\nno\bc-\bc^*\nno\leq 4\sqrt{\constantmargin R(\bc,\bc^*) }$ for any $ \bc\in\cC $. We then have for any $\epsilon>0$:
\begin{align*}
(R^\lambda-R)(\bc,\bc^{*})
&\leq \int_{\bR} \big|\mathcal{K}(u)\big|\frac{L|u|^{s}}{l!}du~\lambda^s\int_{\cB(0,1)}|\gamma(\bc,x)-\gamma(\bc^*,x)|dx\\
&\leq8\int_{\bR} \big|\mathcal{K}(u)\big|\frac{L|u|^{s}}{l!}du~\lambda^s\sqrt{\constantmargin R(\bc,\bc^*)}\\
&\leq \left(8\int_{\bR} \big|\mathcal{K}(u)\big|\frac{L|u|^{s}}{l!}du\right)^2\epsilon\constantmargin \lambda^{2s}+\frac{1}{2\epsilon}R(\bc,\bc^{*}),\\
\end{align*}
where the last inequality comes from Young's inequality:
$$
xy^a\leq x^{1/(1-a)}+ ay,~\forall a<1,~\forall x,y>0,
$$
with $a=\frac{1}{2}$.

The same algebra in the $d$-dimensional case leads to the definition of constant $\constbias$.
\end{proof}

\begin{proposition}[Variance control]
\label{prop deviation massart}
Assume that  $\constantsupconvolution:=\no \eta\no_\infty<\infty $, and {\bf NA}$(\constantFourierPolynomial,\beta)$ and {\bf MA}$(\constantmargin)$ are satisfied for some $ \beta\in(0,\infty)^d $, $\constantFourierPolynomial>0$, and for some positive constant $\constantmargin$, respectively.
Then, for any  $ t,\constseuil>0 $, we have with probability $1-e^{\text-t}$:
$$
\big|R_n^{\lambda}-R^{\lambda}\big|(\bc,\bc^*)\leq \constseuil^{\text-1}\left[R(\bc,\bc^*)+\ordreterme{\lambda}{\constseuil}{t}\right],\,\mbox{ for all }\bc\in\mathcal{C},
$$
where $\ordreterme{\lambda}{\constseuil}{t}$ satisfies the following fixed point equation:
\begin{equation}\label{eq point fixe}
\sqrt r\constseuil{\frac{\lambda^{\text-\bar\beta}}{\sqrt n}}\left[\constsupboundprocess\sqrt{\constantsupconvolution} \sqrt{2\constantmargin t}+6\left(1+\frac{1}{\sqrt n\lambda^{d/2}}\right)\constloustau+\frac{2t  \constsupboundprocess}{3\sqrt{n}\lambda^{d/2}}\right]+\frac{2t}{3n}
=r,
\end{equation}
with $ \bar\beta=\sum_{v=1}^d \beta_v $ and:
$$
\constsupboundprocess:=4\sqrt{\mathcal{V}(d)}\constantFourierPolynomial^{\text-1}\|w\|_1 ~ \text{and}~\constloustau:=\frac{8\sqrt{\mathcal{V}(d)}}{\constantFourierPolynomial}\|w\|_1\left(\frac{1}{3}\vee \sqrt{2\constantsupconvolution}\right)  \left[\log(|\cM|)+kd\big(\log (kd)+6\log(2)\big)\right],
$$
and $\mathcal{V}(d)=\pi^{d/2}/\Gamma(d/2+1)$.
\end{proposition}
\begin{proof}
The proof of Proposition \ref{prop deviation massart} is based on the Bousquet's version of Talagrand concentration inequality (see \cite{Bousquet02}) applied to the following weighted random variable:
\beqn
\label{weightedep}
Z_\lambda:=A\sup_{\bc\in\mathcal{C}}\frac{r}{r+R(\bc,\bc^*)}\left|\sino\gamma_\lambda(\bc,Z_i)-\gamma_\lambda(\bc^*,Z_i)-\bE\big[\gamma_\lambda(\bc,Z)-\gamma_\lambda(\bc^*,Z)\big]\right|.
\eeqn
More precisely, let us introduce the following notations. For any $ \lambda\in\Lambda $, define
\begin{equation}
\label{bousquetterms}
\sigma^2_\lambda:=A^2\sup_{\bc\in\cC}
\bE\left[\frac{r\big(\gamma_\lambda(\bc,Z)-\gamma_\lambda(\bc^*,Z)\big)}{r+R(\bc,\bc^*)}\right]^2,~\op{and}~T_\lambda :=A \sup_{\bc\in\cC}\left\|\frac{r\big[\gamma_\lambda(\bc,\cdot)-\gamma_\lambda(\bc^*,\cdot)\big]}{r+R(\bc,\bc^*)}\right\|_\infty.
\end{equation}
Then, under separability condition over $ \cC $, for any $ t>0 $, the classical Bousquet's inequality claims that:
\begin{equation*}
\bP\left(Z_\lambda\geq \bE Z_\lambda+\sqrt{\frac{2t}{n}\left(\sigma^2_\lambda+\left[1+T_\lambda\right] \bE Z_\lambda\right)}+\frac{t}{3n}\right)\leq e^{\text-t}.
\end{equation*}
In the sequel, we use a simpler version of Bousquet's inequality as follows. By simple algebra, we have:
\begin{align*}
\bE Z_\lambda+\sqrt{\frac{2t}{n}\left(\sigma^2_\lambda+\left[1+T_\lambda\right] \bE Z_\lambda\right)}+\frac{t}{3n}
&\leq\bE Z_\lambda+\sigma_\lambda\sqrt{\frac{2t}{n}}+\sqrt{\frac{2t}{n}\left(1+T_\lambda\right) \bE Z_\lambda}+\frac{t}{3n}\\
&\leq \sigma_\lambda\sqrt{\frac{2t}{n}}+\bE Z_\lambda+2\sqrt{\bE Z_\lambda}\sqrt{\frac{t}{n}\left(1+T_\lambda\right)} +\frac{t\left(1+T_\lambda\right)}{3n}\\
&\leq \sigma_\lambda\sqrt{\frac{2t}{n}}+2\bE Z_\lambda +2\frac{t\left(1+T_\lambda\right)}{3n},
\end{align*}
where we use $ (a+b)^2\leq2(a^2+b^2),\:a,b\in\bR $ to get the last inequality. We can then give a simpler version of Bousquet's inequality.
For any $ t>0 $, we have:
\begin{equation}
\label{bousquetsimplified}
\bP\left(Z_\lambda\geq 2\bE Z_\lambda+\sigma_\lambda\sqrt{\frac{2t}{n}}+ 2\frac{t\left(1+T_\lambda\right)}{3n}\right)\leq e^{\text-t}.
\end{equation}

\medskip
We are now on time to give convenient upper bounds for the terms depending on $\lambda$ in \eqref{bousquetsimplified}.

\smallskip
\noindent{\it Control of $ \sigma_\lambda $:} The control of $ \sigma_\lambda $ is based on Lemma \ref{lipappli} and the margin assumption {\bf MA}$(\constantmargin)$. We have, from \eqref{bousquetterms}:
\begin{equation}\label{eq control variance}
\sigma_\lambda^2=A^2\sup_{\bc\in\cC}\frac{r^2\:\bE\big(\gamma_{\lambda}(\bc,Z)-\gamma_{\lambda}(\bc^*,Z)\big)^2}{\left(R(\bc,\bc^*)+r\right)^2}\leq A^2 \sup_{\bc\in\cC}\frac{r^2\constsupboundprocess^2\constantsupconvolution \lambda^{\text-2\bar\beta}\constantmargin R(\bc,\bc^*)}{\left(R(\bc,\bc^*)+r\right)^2}\leq {\constseuil^2\constantmargin\constsupboundprocess^2\constantsupconvolution \lambda^{\text-2\bar\beta}r},
\end{equation}
where the last inequality is obtained considering both cases $ R(\bc,\bc^*)\leq r $ and $ R(\bc,\bc^*)>r $.
\\

\smallskip
\noindent{\it Control of $ T_\lambda $:}
Using Lemma \ref{lem sup bound process}, and the margin assumption {\bf MA}$(\constantmargin)$, one gets immediately:
\begin{equation}\label{eq control sup borne}
T_\lambda\leq \constseuil \sqrt{\constantmargin r} \constsupboundprocess\lambda^{-\bar\beta-d/2}.
\end{equation}

\smallskip
\noindent{\it Control of $ \bE Z_\lambda $:} The control of $ \bE Z_\lambda $ needs Lemma \ref{complexity} and a modified version of the so-called {\it peeling device} suggested by \cite{levrard}. Let $\mu>1$ be a real number. We then have:
$$
Z_\lambda\leq \sup_{\bc\in\cC \::\:R(\bc,\bc^*)\leq r}\frac{\constseuil r\:\big|R_n^{\lambda}-R^{\lambda}\big|(\bc,\bc^*)}{R(\bc,\bc^*)+r}+\sum_{k\geq0} A\sup_{\bc\in\cC \::\:r\mu^k\leq R(\bc,\bc^*)\leq r\mu^{k+1}}\frac{ r\:\big|R_n^{\lambda}-R^{\lambda}\big|(\bc,\bc^*)}{R(\bc,\bc^*)+r}
$$
Taking the expectation on both sides and using Lemma \ref{complexity} lead to:
\begin{align*}
\bE\bZ_\lambda&\leq\frac{\constseuil\constloustau\lambda^{\text-\bar\beta}\sqrt r }{\sqrt n}\left(1+\frac{1}{\sqrt n\lambda^{d/2}}\right)+\sum_{k\geq0}\frac{\constseuil\constloustau\lambda^{\text-\bar\beta}\sqrt r \mu^{(k+1)/2}}{\sqrt n(1+\mu^k)}\left(1+\frac{1}{\sqrt n\lambda^{d/2}}\right)\\
&\leq\frac{\constseuil\constloustau\lambda^{\text-\bar\beta} \sqrt r}{\sqrt{n}}\left(1+\frac{1}{\sqrt n\lambda^{d/2}}\right)\left(1+\sum_{k\geq0}\frac{ \mu^{(k+1)/2}}{(1+\mu^k)}\right)\leq 3\constseuil\constloustau\frac{\lambda^{\text-\bar\beta}\sqrt r}{\sqrt{n}}\left(1+\frac{1}{\sqrt n\lambda^{d/2}}\right),
\end{align*}
where the last inequality is obtained taking $ \mu=4. $
Thus, using the definition of $ r=\ordreterme{\lambda}{\constseuil}{t}$, the last inequality, \eqref{eq control variance} and \eqref{eq control sup borne}, we have 
\begin{align*}
\sigma_\lambda\sqrt{\frac{2t}{n}}+2\bE \bZ_\lambda +2\frac{t\left(1+T_\lambda\right)}{3n}
&\leq\sqrt r\constseuil{\frac{\lambda^{\text-\bar\beta}}{\sqrt n}}\left[\constsupboundprocess\sqrt{\constantsupconvolution} \sqrt{2\constantmargin t}+6\left(1+\frac{1}{\sqrt n\lambda^{d/2}}\right)\constloustau+\frac{2t  \constsupboundprocess}{3\sqrt{n}\lambda^{d/2}}\right]+\frac{2t}{3n}\\
&=r=\ordreterme{\lambda}{\constseuil}{t}.
\end{align*}
Invoking Bousquet's inequality \eqref{bousquetsimplified} and the last inequality, we obtain
$
\bP\left(Z_\lambda\geq \ordreterme{\lambda}{\constseuil}{t}\right)\leq e^{\text-t}.
$
Then, the proof is complete by definition of $ Z_\lambda $ in \eqref{weightedep}.
\end{proof}

\subsection{Proof of Theorem \ref{thm:nonadaptiveup}}
As we mentioned above, the proof is based on the application of Propositions \ref{prop bias} and \ref{prop deviation massart}. Indeed, from the bias-variance decomposition \eqref{bvdecomposition}, using Proposition \ref{prop bias} with $ \epsilon=2 $ and $ \bc=\hat{\mathbf{c}}_\lambda$, and Proposition \ref{prop deviation massart} with $ A=4 $ and $ \bc=\hat{\mathbf{c}}_\lambda$, we have with probability $ 1-e^{-t} $
$$
R(\hat\bc_{\lambda},\bc^{*})\leq2\constbias\constantmargin \lambda^{2s}+\frac{1}{2}R(\hat\bc_{\lambda},\bc^{*})+\frac{1}{4}r^*_{\lambda,n}(4,t)
$$
Moreover, note that for any $\lambda>0$ such that $ n\lambda^d\to \infty $ as $ n\rightarrow\infty $, and for any $ t>0$ and any $\constseuil\geq 1 $, there exists $ n_0\in\bN $ such that for all $n\geq n_0$ it holds:
\begin{equation}\label{eq sup point fixe}
\ordreterme{\lambda}{\constseuil}{t}\leq 2A^2\frac{\lambda^{\text-2\bar\beta}}{ n}\left[\constsupboundprocess\sqrt{\constantsupconvolution} \sqrt{2\constantmargin t}+6\left(1+\frac{1}{\sqrt n\lambda^{d/2}}\right)\constloustau+\frac{2t  \constsupboundprocess}{3\sqrt{n}\lambda^{d/2}}\right]^2,
\end{equation}
looking at the solution of the fixed point equation \eqref{eq point fixe}. Using two last inequalities, we have for $n$ such that $n\lambda^d\geq 1$,  with probability $ 1-e^{-t}$
\begin{equation}\label{eq sup bound excess risk}
R(\hat\bc_{\lambda},\bc^{*})\leq4\constbias\constantmargin \lambda^{2s}+64\frac{\lambda^{\text-2\bar\beta}}{ n}\left[\constsupboundprocess^2\constantsupconvolution 2\constantmargin t+72\constloustau^2+\frac{4t^2  \constsupboundprocess^2}{9n\lambda^{d}}\right].
\end{equation}
By integration, the choice of $\lambda$ in Theorem \ref{thm:nonadaptiveup} gives:
$$
\bE R(\hat\bc_{\lambda},\bc^{*})\leq \max\left(4\constbias\constantmargin, 48\left(2\constsupboundprocess^2\constantsupconvolution\constantmargin+12\constloustau^2\right)\right)n^{-s/(s+\bar\beta)}.
$$
\epr

\subsection{Proof of Theorem \ref{th adaptivity lambda}} In the sequel, we write $\lambda^*$ the element of $\Lambda$ solution of the equation $ 12\constbias\kappa \lambda^{2s}=\threshold\lambda $, where $ \constbias $ is given in Proposition \ref{prop bias}. Consequently, note that there exists a constant $ \Diamond>0 $ such that 
$$
\lambda^*=\Diamond n^{-\frac{1}{2s+2\bar\beta}}.
$$
Moreover, we introduce $ \lambda^*_a $, the element of $ \Lambda_a $ such that $ \lambda^*_a\leq\lambda^*\leq a^{\text-1}\lambda^*_a $.
Let us consider the event $ \Omega=\big\{\lambda^*_a\leq\hat\lambda\big\} $. Firstly, by construction of $ \hat\lambda $, we have on $\Omega$:
\begin{equation}\label{eq construction}
R_n^{\lambda^*_a}(\hat\bc_{\hat \lambda},\hat\bc_{\lambda^*_a})=R_n^{\lambda^*_a}(\hat\bc_{\hat \lambda})-R_n^{\lambda^*_a}(\hat\bc_{\lambda^*_a})\leq 3\threshold{\lambda^*_a}.
\end{equation}
By simple computations, one gets:
\beqnn
 R(\hat\bc_{\hat \lambda},\bc^*)&=&(R-R^{\lambda^*_a})(\hat\bc_{\hat \lambda},\bc^{*})+(R^{\lambda^*_a}-R_n^{\lambda^*_a})(\hat\bc_{\hat \lambda},\bc^{*})+R_n^{\lambda^*_a}(\hat\bc_{\hat \lambda},\hat\bc_{\lambda^*_a})+R_n^{\lambda^*_a}(\hat\bc_{\lambda^*_a},\bc^{*})\\
&\leq&(R-R^{\lambda^*_a})(\hat\bc_{\hat \lambda},\bc^{*})+(R^{\lambda^*_a}-R_n^{\lambda^*_a})(\hat\bc_{\hat \lambda},\bc^{*})+R_n^{\lambda^*_a}(\hat\bc_{\hat \lambda},\hat\bc_{\lambda^*_a}).
\eeqnn
Then, using Proposition \ref{prop bias} with some $ \epsilon>1/2$ (chosen later on) and $ \bc=\hat\bc_{\hat \lambda} $, it yields on $\Omega$ using \eqref{eq construction}:
\begin{eqnarray*}
R(\hat\bc_{\hat \lambda},\bc^*)\leq
\epsilon\constbias\constantmargin (\lambda^*_a)^{2s}+\frac{1}{2\epsilon}R(\hat\bc_{\hat \lambda},\bc^*)+(R^{\lambda^*_a}-R_n^{\lambda^*_a})(\hat\bc_{\hat \lambda},\bc^{*})+3\threshold{\lambda^*_a}.
\end{eqnarray*}
Using
Proposition \ref{prop deviation massart} with $t=2\log(n) $, and $ \bc=\hat\bc_{\hat \lambda}$, it holds for any $A>2\epsilon/(2\epsilon-1)$, with probability at least $1-n^{\text-2}-\P(\lambda^*_a\geq\hat{\lambda})$:
$$
R(\hat\bc_{\hat \lambda},\bc^*)\leq
\frac{2\epsilon A}{2\epsilon A-(2\epsilon+A)}\left(\epsilon\constbias\constantmargin (\lambda^*_a)^{2s}+\frac{1}{A}\ordreterme{\lambda^*_a}{A}{2\log(n)}+3\threshold{\lambda^*_a}\right).
$$
Choosing $\epsilon=2$ and $A=\sqrt{2}$, we obtain:
$$
R(\hat\bc_{\hat \lambda},\bc^*)\leq
2\left(2\constbias\constantmargin (\lambda^*_a)^{2s}+\frac{1}{\sqrt{2}}\ordreterme{\lambda^*_a}{\sqrt{2}}{2\log(n)}+3\threshold{\lambda^*_a}\right).
$$
From the definition of $ \delta_\lambda $ and \eqref{eq sup point fixe}, there exists $ n_1\in\bN $ such that $ \threshold{\lambda}\geq \ordreterme{\lambda}{\sqrt{2}}{2\log(n)} $ for any $ \lambda\in\Lambda$ and $ n\geq n_1 $. We then obtain:
\begin{eqnarray}\label{eq step1 lepski}
R(\hat\bc_{\hat \lambda},\bc^*)\leq
4\constbias\constantmargin (\lambda^*_a)^{2s}+8\threshold{\lambda^*_a},
\end{eqnarray}
\\

It remains to control the probability $ \bP(\hat\lambda\leq \lambda^*_a) $. Note that, using the definition of $ \hat\lambda $, we have by union bound:
\begin{align}\label{eq decomposition cas 2}
\P(\lambda^*_a\geq\hat{\lambda})\leq\sum_{\lambda\leq\lambda^*_a}\P\left(R_n^{\lambda}(\estimator{\lambda^*_a},\hat\bc_{\hat \lambda})>3\threshold{\lambda}\right).
\end{align}
From simple computations and using twice Proposition \ref{prop bias} with $ \epsilon>0 $ and $ \bc=\hat{\mathbf{c}}_\lambda$ or $\bc=\hat\bc_{\lambda^*_a}$, for any $ \lambda\leq \lambda^*_a $, we have:
\begin{align*}
R_n^{\lambda}(\hat\bc_{\lambda^*_a},\hat\bc_{\hat \lambda})&\leq(R_n^{\lambda}-R^\lambda)(\hat\bc_{\lambda^*_a},\hat\bc_{\hat \lambda})+(R^{\lambda}-R)\big(\hat\bc_{\lambda^*_a},\hat\bc_{\lambda}\big)+R(\hat\bc_{\lambda^*_a},\hat\bc_{\lambda})\\
&\leq(R_n^{\lambda}-R^\lambda)(\hat\bc_{\lambda^*_a},\hat\bc_{\lambda})+2\constbias \constantmargin \epsilon\lambda^{2s}+\frac{1}{2\epsilon}R(\hat\bc_{\lambda^*_a},{\bc}^*)+\frac{1}{2\epsilon}R(\hat\bc_{\lambda},{\bc}^*)+R(\hat\bc_{\lambda^*_a},\hat\bc_{\lambda})\\
&\leq(R_n^{\lambda}-R^\lambda)(\hat\bc_{\lambda^*_a},\hat\bc_{\lambda })+2\constbias\constantmargin \epsilon\lambda^{2s}+\left(1+\frac{1}{2\epsilon}\right)R(\hat\bc_{\lambda^*_a},{\bc}^*)+\left(\frac{1}{2\epsilon}-1\right)R(\hat\bc_{\lambda^*_a},{\bc}^*).
\end{align*}
Note that $ \threshold{\lambda}\geq\sqrt{2}\ordreterme{\lambda}{\sqrt{2}}{2\log(n)} $ for any $ \lambda\in\Lambda $ and $ n\geq n_1 $. Then using twice Proposition \ref{prop deviation massart} with $ A=\sqrt{2} $ and $ \bc=\estimator{\lambda^*_a}$ or $ \bc=\estimator{\lambda}$, we get with probability $ 1-2n^{\text-2} $
$$
R_n^{\lambda}(\hat\bc_{\lambda^*_a},\hat\bc_{\lambda })\leq\threshold\lambda+2\constbias\constantmargin \epsilon\lambda^{2s}+\left(1+\frac{1}{2\epsilon}+\frac{1}{\sqrt{2}}\right)R(\hat\bc_{\lambda^*_a},{\bc}^*)+\left(\frac{1}{2\epsilon}-1+\frac{1}{\sqrt{2}}\right)R(\hat\bc_{\lambda},{\bc}^*).
$$
We are now on time to choose $\epsilon=1/(2-\sqrt{2})$ to get with probability $ 1-2n^{\text-2} $
$$
R_n^{\lambda}(\hat\bc_{\lambda^*_a},\hat\bc_{\lambda })\leq\threshold\lambda+4\constbias\constantmargin \lambda^{2s}+2R(\hat\bc_{\lambda^*_a},{\bc}^*).
$$ 
Finally, using the inequality \eqref{eq sup bound excess risk} with $ \lambda\leq\lambda^*_a $ and $t=2\log(n) $, we get with probability $ 1-3n^{\text-2} $
$$
R_n^{\lambda}(\hat\bc_{\lambda^*_a},\hat\bc_{\lambda })\leq4\constbias\constantmargin \lambda^{2s}+\threshold{\lambda}+8\constbias\constantmargin (\lambda^*_a)^{2s}+\threshold{\lambda^*_a}.
$$
By definition of $\lambda^*_a$ and $ \lambda^* $, it holds that $ \lambda\leq{\lambda^*_a}\leq \lambda^* $ and $ 12\constbias\constantmargin (\lambda^*_a)^{2s}\leq 12\constbias\constantmargin (\lambda^*)^{2s}=\threshold{\lambda^*}\leq\threshold{\lambda^*_a}\leq \threshold{\lambda}$, we thus get with probability $ 1-3n^{\text-2} $
$$
R_n^{\lambda}(\hat\bc_{\lambda^*_a},\hat\bc_{\lambda})\leq3\threshold\lambda.
$$
Using \eqref{eq decomposition cas 2} and the last inequality, we finally get:
$$
\P(\lambda^*_a\geq\hat{\lambda})\leq|\Lambda_a|/n^2\leq 3/n.
$$
Thus, using \eqref{eq step1 lepski}, the last inequality, and definitions of $ \lambda^*_a $ and $ \lambda^* $, there exists a universal constant $ \constadaptative $ depending on $ \constantmargin,w, L, d,s, \beta, \constantFourierPolynomial, k, \constantsupconvolution,$ and $ |\cM|  $ such that with probability greater than $1-4n^{\text-1}$:
$$
R(\hat\bc_{\hat \lambda},\bc^*)\leq
4\constbias\constantmargin (\lambda^*_a)^{2s}+7\threshold{\lambda^*_a}\leq4\constbias \constantmargin (\lambda^*)^{2s}+7a^{\text-2\bar\beta}\threshold{\lambda^*}\leq\constadaptative \left(\frac{\log(n)}{n}\right)^{s/(s+\bar\beta)}.
$$
This last assertion allows to complete the proof taking $ n\geq n_1 $ sufficiently large such that $ 4n^{\text-1} $ is negligible in comparison to $\constadaptative \left(\frac{\log(n)}{n}\right)^{s/(s+\bar\beta)}$.
\epr

\section{Appendix}\label{sec:appendix}

In the proofs of Proposition \ref{prop bias} and Proposition \ref{prop deviation massart}, we use the following technical results.
              
\begin{lemma}
\label{lipappli}
Assume that {\bf NA}$(\constantFourierPolynomial,\beta)$ is satisfied for some $ \beta\in(\frac{1}{2},\infty)^d $ and $\constantFourierPolynomial>0$; and assume $\constantsupconvolution:=\no \eta\no_\infty<\infty $. Then, for any $\bc,\bc'\in\cC^2$, it holds
$$
 \left(\expectationZ\big[\gamma_\lambda(\bc,Z)
 -\gamma_\lambda(\bc',Z)\big]^2\right)^{1/2}
 \leq\constsupboundprocess\sqrt{\constantsupconvolution}\lambda^{-\bar\beta}\nno \bc-\bc'\nno,
$$
where $\constsupboundprocess=4\sqrt{\mathcal{V}(d)}\constantFourierPolynomial^{\text-1}\|w\|_1$  and $ \bar\beta=\sum_{v=1}^d \beta_v $.
\end{lemma}   
\begin{proof}

For any $x\in\R^d$, let $m(x)=[\gamma(\bc,x)-\gamma(\bc',x)] \ind_{\mathcal{B}(0,1)}(x) $. Using the Plancherel theorem and the convolution, it yields:
\begin{align*}
\expectationZ\big(\gamma_\lambda(\bc,Z)-\gamma_\lambda(\bc',Z)\big)^2&=\expectationZ\big([\cK_{\lambda}\ast m](Z)\big)^2\\
&\leq \constantsupconvolution\int_{\R^d} \big([\cK_{\lambda}\ast m](z)\big)^2dz\\
&=\constantsupconvolution\int_{\R^d} \big|\cF[\cK_{\lambda}](t)\big|^2\big|\cF[m](t)\big|^2dt.
\end{align*}
Let us bound the Fourier transform of $ \noyaudeconvolution{\lambda}{\cdot} $; by definition of $ w $, using {\bf NA}$(\constantFourierPolynomial,\beta)$ and Riemann-Lebesgue Theorem, we have:
\begin{align}\label{eq bound fourier kernel}
 \big|\cF[\cK_{\lambda}](t)\big|^2=\left|\frac{\cF[w](\lambda t)}{\cF[\eta](t)}\right|^2
&\leq\sup_{t\in[\text-\lambda^{\text-1},\lambda^{\text-1}]}\frac{\|w\|_1^2}{\left|\cF[\eta](t)\right|^2}\nonumber\\
&\leq \constantFourierPolynomial^{-2}\|w\|_1^2\sup_{t\in[\text-\lambda^{\text-1},\lambda^{\text-1}]^d} \left(\frac{t^2+1}{2}\right)^{\beta}\leq \constantFourierPolynomial^{-2}\|w\|_1^2 \lambda^{\text-2\bar\beta}.
\end{align}
Two last inequalities then imply:
$$
\expectationZ\big(\gamma_\lambda(\bc,Z)-\gamma_\lambda(\bc',Z)\big)^2\leq \constantsupconvolution\constantFourierPolynomial^{-2}\|w\|_1^2 \lambda^{\text-2\bar\beta}\int \big|\cF[m](t)\big|^2dt= \constantsupconvolution\constantFourierPolynomial^{-2}\|w\|_1^2 \lambda^{\text-2\bar\beta}\int |m(x)|^2dx.
$$
By definition of $m$, for any $ \bc,\bc'\in\cC^2 $, we get
\begin{equation}\label{eq bound m}
 |m(x)|\leq\max_{j=1,\dots,k}\big|\no x-c_j\no^2-\no x-c_j'\no^2\big|\leq\max_{j=1,\dots,k}\no c_j-c'_j\no\:\big(\no x-c_j\no+\no x-c_j'\no\big)\leq 4 \nno\bc-\bc'\nno.
\end{equation}
Therefore, \eqref{eq bound m} and \eqref{eq bound fourier kernel} yield:
$$
\expectationZ\big(\gamma_\lambda(\bc,Z)-\gamma_\lambda(\bc',Z)\big)^2\leq 16\mathcal{V}(d)\constantsupconvolution\constantFourierPolynomial^{-2}\|w\|_1^2 \lambda^{\text-2\bar\beta} \nno\bc-\bc'\nno^2,
$$
where $\mathcal{V}(d)=\frac{\pi^{d/2}}{\Gamma(d/2+1)}$ is the Lebesgue measure on $\R^d$ of $\mathcal{B}(0,1)$.
\end{proof}
         
\begin{lemma}
\label{lem sup bound process}
Assume that {\bf NA}$(\constantFourierPolynomial,\beta)$ is satisfied for some $ \beta\in(\frac{1}{2},\infty)^d $ and $\constantFourierPolynomial>0$. Then, for any $\bc,\bc'\in\cC^2$, it holds
$$
\|\gamma_\lambda(\bc,\cdot)-\gamma_\lambda(\bc',\cdot)\|_\infty
\leq\constsupboundprocess\lambda^{-\bar\beta-d/2}\nno \bc-\bc'\nno,
$$
where $\constsupboundprocess=4\sqrt{\mathcal{V}(d)}\constantFourierPolynomial^{\text-1}\|w\|_1$ and $ \bar\beta=\sum_{v=1}^d \beta_v $.
\end{lemma}
\begin{proof}
As in the proof of Lemma \ref{lipappli}, let $ m(\cdot) $ denote the function $ [\gamma(\bc,\cdot)-\gamma(\bc',\cdot)] \ind_{\mathcal{B}(0,1)}(\cdot) $. Using Cauchy-Schwarz inequality, it yields 
$$
\|\gamma_\lambda(\bc,\cdot)-\gamma_\lambda(\bc',\cdot)\|_\infty= \sup_{z\in\bR^d}\left|\int \noyaudeconvolution{\lambda}{z-x} m(x)dx\right|\leq \sup_{z\in\bR^d}\sqrt{\int \big[\noyaudeconvolution{\lambda}{z-x}\big]^2dx\int m^2(x)dx}.
$$
From \eqref{eq bound m}, we have that $\int m^2(x)dx\leq 16\mathcal{V}(d)\no\bc-\bc'\no^2$. Moreover, using the definition of $ \noyaudeconvolution{\lambda}{\cdot}$ and some change of variables, we get with {\bf NA}$(\constantFourierPolynomial,\beta)$:
$$
\sup_{z\in\bR^d}\int\big[\noyaudeconvolution{\lambda}{z-x}\big]^2dx=\int \big|\cF[\cK_{\lambda}](t)\big|^2dt=\int\left|\frac{\cF[w](\lambda t)}{\cF[\eta](t)}\right|^2dt\leq\constantFourierPolynomial^{-2}\|w\|_1^2\lambda^{-2\bar\beta-d}.
$$
Three last inequalities then imply 
$$
\|\gamma_\lambda(\bc,\cdot)-\gamma_\lambda(\bc',\cdot)\|_\infty\leq 4\sqrt{\mathcal{V}(d)}\constantFourierPolynomial^{-1}\|w\|_1\lambda^{-\bar\beta-d/2}\nno\bc-\bc'\nno.
$$
\end{proof}
With this result, we also need to control the complexity involved in Section \ref{sec:nonadaptive} thanks to the following lemma:
\begin{lemma}
\label{complexity}
Assume that {\bf NA}$(\constantFourierPolynomial,\beta)$ are satisfied for some $ \beta\in(\frac{1}{2},\infty)^d $ and $\constantFourierPolynomial>0$. Then, $\forall \lambda,\delta>0$, we have: 
$$
\mathbb{E} \sup_{(\mathbf{c},\mathbf{c}^*)\in\mathcal{C}\times\mathcal{M}, \|\mathbf{c}-\mathbf{c}^*\|^2 \leq \delta} {\big|R_n^{\lambda}-R^{\lambda}\big|(\bc,\bc^*)} \leq \constloustau\frac{\sqrt\delta}{\sqrt n \lambda^{\bar\beta}}\left(1+\frac{1}{\sqrt n\lambda^{d/2}}\right),
$$
where 
$$
\constloustau:=\frac{8\sqrt{\mathcal{V}(d)}}{\constantFourierPolynomial}\|w\|_1\left(\frac{1}{3}\vee \sqrt{2\constantsupconvolution}\right)  \left[\log(|\cM|)+kd\big(\log (kd)+6\log(2)\big)\right].
$$
\end{lemma} 
\begin{proof} The proof is based on the chaining argument, gathering with a maximal inequality. More precisely, we use a special case of a maximal
inequality derived in \citep[Lemma 6.6, Section 6.1]{Massart07}. Adapted to our
needs, this maximal inequality reads as follows: 
\begin{lemma}[Maximal Inequality]
\label{maximal}
Let $ \cX_1,\dots,\cX_n $ be a sequence of independent random variables. For any finite subset $ \Phi $ of real functions, assume there exist some constants $ \sigma,b>0 $ such that for any $ \phi\in\Phi $
$$
\sino \bE\phi^2(\cX_i)\leq \sigma^2\text{~and~} \|\phi\|_\infty\leq b.
$$
Then: 
\begin{equation}\label{eq maximal inequality}
\bE\sup_{\phi\in\Phi}\left|\sino\phi(\cX_i)-\bE\phi(\cX_i)\right|\leq \frac{2\sigma}{\sqrt n}\sqrt{2\log(|\Phi|)}+\frac{2b}{3n}\log(|\Phi|),
\end{equation}
where $ |\Phi| $ denotes the cardinal of the set $ \Phi $.
\end{lemma}
 
Now, we start with the main part of the proof of the lemma. Consider $ L:\bc\in\cC\mapsto L(\bc)= \gamma_\lambda(\bc,Z)-\gamma_\lambda(\mathbf{c}^*,Z)$. We will use a chaining argument on the function $L$ as follows. For some $\delta>0$, let $ \bc^0\in\cC $ be fixed such that $ \nno\mathbf{c}^0-\mathbf{c}^*\nno^2 \leq \delta $. For some $0<a<1$, for any $ v\in\bN^* $, denote $ \Gamma_v $ a $ \delta a^{v} $-net of $ \cC $. For $ \bc\in\cC $, introduce the following notations:
$$
u_0(\bc):=\bc^0,~~ u_v(\bc):=\arg\inf_{u\in\Gamma_v}\nno u-\bc\nno,~~v\in\bN^*.
$$
We then deduce that $ u_v(\bc)\to\bc $ for $ v\to\infty $. The main ingredient of the chaining argument is the following decomposition (by continuity of $L$ using dominated convergence theorem):
$$
L(\bc)=L(\bc^0)+\sum_{v\in\bN^*}L(u_v(\bc))-L(u_{v-1}(\bc)).
$$
Note that for simplicity, in the sequel, we write $\bc^0$ instead of $\bc^0(\bc^*)$ since this vector depends on the minimum $\bc^*\in\mathcal{M}$ where the cardinality of $\mathcal{M}$ could satisfy $|\mathcal{M}|\geq 2$.
For easy of exposition, denote $ P_n $ the empirical measure and $ P $ the expectation w.r.t. the distribution $ \probaZ $, we then obtain:
\begin{align}\label{eq chaining}
\mathbb{E}& \sup_{(\mathbf{c},\mathbf{c}^*)\in\mathcal{C}\times\mathcal{M}, \|\mathbf{c}-\mathbf{c}^*\|^2\leq \delta} {\big|R_n^{\lambda}-R^{\lambda}\big|(\bc,\bc^*)}\nonumber\\
&=\mathbb{E}  \sup_{\mathbf{c}^* \in \mathcal{M}}\sup_{\bc\in\cC\::\:\|\mathbf{c}-\mathbf{c}^*\|^2 \leq \delta} {\left|(P_n-P)(\gamma_\lambda(\mathbf{c}^*,Z) - \gamma_\lambda(\mathbf{c},Z))\right|}\nonumber\\
&\leq\mathbb{E}\sup_{\mathbf{c}^* \in \mathcal{M}} \left|(P_n-P)\big(\gamma_\lambda(\bc^0,Z)-\gamma_\lambda(\mathbf{c}^*,Z)\big)\right|\nonumber\\
&~~+\mathbb{E}\sup_{\mathbf{c}^* \in \mathcal{M}} \sup_{\bc\in\cC\::\:\|\mathbf{c}-\mathbf{c}^*\|^2 \leq \delta}\sum_{v\in\bN^*}\left|(P_n-P)\big(\gamma_\lambda(u_v(\bc),Z)-\gamma_\lambda(u_{v-1}(\bc),Z)\big)\right|\nonumber\\
&\leq\mathbb{E}\sup_{\mathbf{c}^*\in\mathcal{M}} \left|(P_n-P)\big(\gamma_\lambda(\bc^0,Z)-\gamma_\lambda(\mathbf{c}^*,Z)\big)\right|\nonumber\\
&~~+\sum_{v\in\bN^*}\mathbb{E}\sup_{(u,u')\in\Gamma_v\times\Gamma_{v\text-1}\::\:\|u-u'\|^2 \leq \delta a^v}\left|(\empiricalproba-P)\big(\gamma_\lambda(u,Z)-\gamma_\lambda(u',Z)\big)\right|\nonumber\\
&=:A_1+A_2
\end{align}

We will now find bounds of $ A_1 $ and $ A_2 $ thanks to Lemma \ref{maximal}.

\paragraph{Bound of $A_1$:} 
We first remind that $|\mathcal{M}|$ is finite. We can then apply Lemma \ref{maximal}, with $ \phi(\cX_i)=\gamma_\lambda(\bc^0,Z_i)-\gamma_\lambda(\mathbf{c}^*,Z_i) $. Indeed, using the definition of $ \bc^0 $ and Lemmas \ref{lipappli} and \ref{lem sup bound process}, we have $ \sigma=\constsupboundprocess\sqrt{\constantsupconvolution}\lambda^{-\bar\beta}\sqrt\delta $ and $ b = \constsupboundprocess\lambda^{-\bar\beta-d/2}\sqrt\delta $ in Lemma \ref{maximal}, it yields:
$$
A_1=\mathbb{E}\sup_{\mathbf{c}^*\in\mathcal{M}} \left|(P_n-P)\big(\gamma_\lambda(\bc^0,Z)-\gamma_\lambda(\mathbf{c}^*,Z)\big)\right|\leq 2\constsupboundprocess\sqrt{\constantsupconvolution}\sqrt{2\log(|\cM|)}\frac{\lambda^{-\bar\beta}\sqrt\delta}{\sqrt n}
+\frac{2\constsupboundprocess\log(|\cM|)}{3}\frac{\lambda^{-\bar\beta-d/2}\sqrt\delta}{n}.
$$

\paragraph{Bound of $A_2$:} 
As previously, we use the maximal inequality with $ \phi(\cX_i)=\gamma_\lambda(u,Z_i)-\gamma_\lambda(u',Z_i) $ to the finite set $ \Phi=\Gamma_v\times\Gamma_{v\text-1} $. According to Lemmas \ref{lipappli} and \ref{lem sup bound process}, we have $ \sigma=\constsupboundprocess\sqrt{\constantsupconvolution}\lambda^{-\bar\beta}\sqrt{a^{v}\delta}  $ and $ b = \constsupboundprocess\lambda^{-\bar\beta-d/2}\sqrt{a^{v}\delta} $. We then have for any $ v\in\bN^* $
\begin{align*}
\mathbb{E}&\sup_{(u,u')\in\Gamma_v\times\Gamma_{v\text-1}\::\:\|u-u'\|^2 \leq \delta a^v}\left|(P_n-P)\big(\gamma_\lambda(u,Z)-\gamma_\lambda(u',Z)\big)\right|\\
&\leq 2\constsupboundprocess\sqrt{\constantsupconvolution} a^{v/2}\sqrt{2\log(|\Gamma_v||\Gamma_{v\text-1}|)}\frac{\lambda^{-\bar\beta}\sqrt\delta}{\sqrt n}
+\frac{2\constsupboundprocess a^{v/2}\log(|\Gamma_v||\Gamma_{v\text-1}|)}{3}\frac{\lambda^{-\bar\beta-d/2}\sqrt\delta}{n}.
\end{align*}
We note that $ |\Gamma_v|=\left(\sqrt{kd}/a^v\right)^{-kd} $ and taking $a=1/4$, we obtain $ \sum_{v\in\bN^*}a^{v/2}\log(|\Gamma_v||\Gamma_{v\text-1}|)\leq kd[\log (kd)+6\log 2]=:\zeta_4 $. We then have from the definition of $ A_2 $ and the last inequality:
$$
A_2\leq 2\sqrt{2}\constsupboundprocess\sqrt{\constantsupconvolution}\zeta_4\frac{\lambda^{-\bar\beta}\sqrt\delta}{\sqrt n}
+2\constsupboundprocess\zeta_4\frac{\lambda^{-\bar\beta-d/2}\sqrt\delta}{3n}.
$$
From last bounds of $ A_1 $ and $ A_2 $, and the chaining decomposition \eqref{eq chaining}, it yields
\begin{align*}
\bE &\sup_{\mathbf{c}^* \in \mathcal{M}}\sup_{\bc\in\cC\::\:\|\mathbf{c}-\mathbf{c}^*\|^2 \leq \delta} {|\empiricalproba-P|(\gamma_\lambda(\mathbf{c}^*,Z) - \gamma_\lambda(\mathbf{c},Z))}\\
&\leq2\sqrt{2}\constsupboundprocess\sqrt{\constantsupconvolution}\big(\sqrt{\log (|\mathcal{M}|)}+\zeta_4\big)\frac{\sqrt\delta}{\sqrt n \lambda^{\bar\beta}}
+\frac{2\constsupboundprocess }{3}\big(\log (|\mathcal{M}|)+\zeta_4\big)\frac{\sqrt\delta}{n\lambda^{\bar\beta+d/2}}
\end{align*}
The proof is complete by definition of $ \constloustau $.
\epr

% \begin{definition}
% \label{app}
% The class $\{\gamma_\lambda(g),\bc\in\cC \}$  has approximation function $a(\lambda)$ and residual constant $0<r<1$ if the following holds:
% \beqnn
% \forall \bc\in\cC ,\,(R-R^\lambda_l)(g,g^{*})\leq a(\lambda)+r(R(g)-R(g^*)),
% \eeqnn
% where with a slight abuse of notations, we write:
% $$
% (R-R^\lambda_\gamma)(g,g^{*})=R(g)-R(g^*)-R^\lambda(g)+R^\lambda(g^*).
% $$
% \end{definition}

\end{proof}

{
\bibliographystyle{plainnat}
\bibliography{reference2}
}
\end{document}